\documentclass[a4paper,11pt]{article}
\usepackage{authblk}
\usepackage{graphicx}

\usepackage{amsmath,amssymb,latexsym,natbib}
\usepackage{url} 

\usepackage[left=3cm,right=2cm,top=3cm,bottom=2cm]{geometry}

\usepackage{color}

\newcommand{\lfis}{{\bf LFI}s}
\newcommand{\lfi}{{\bf LFI}}

\newcommand{\rmbc}{{\bf RmbC}}
\newcommand{\mbc}{{\bf mbC}}
\newcommand{\rmbcciw}{{\bf RmbCciw}}
\newcommand{\qmbc}{{\bf QmbC}}
\newcommand{\rqmbc}{{\bf RQmbC}}
\newcommand{\cplp}{\text{\bf CPL$^+$}}

\newcommand{\cpl}{\text{\bf CPL}}

\newcommand{\sent}{$Sen(\Omega)$}
\newcommand{\tert}{$Ter(\Omega)$}
\newcommand{\fort}{$For_1(\Omega)$}
\newcommand{\afor}{$AFor(\Omega)$}
\newcommand{\ctert}{$CTer(\Omega)$}

\newcommand{\mbcciw}{{\bf mbCciw}}

\newcommand{\bc}{{\bf bC}}
\newcommand{\rbc}{{\bf RbC}}
\newcommand{\ci}{{\bf Ci}}
\newcommand{\rci}{{\bf RCi}}

\newcommand{\cila}{{\bf Cila}}
\newcommand{\rcila}{{\bf RCila}}

\newcommand{\lfium}{{\bf LFI1}}

\newcommand{\dacdot}{{\bf J3}}
\newcommand{\set}{{\bf P1}}

\newcommand{\lptz}{{\bf LPT0}}

\newcommand{\A}{\ensuremath{\mathcal{A}}}
\newcommand{\B}{\ensuremath{\mathcal{B}}}
\newcommand{\matM}{\ensuremath{\mathcal{M}}}
\newcommand{\matF}{\ensuremath{\mathcal{F}}}

\newcommand{\axwci}{{\bf ciw}}
\newcommand{\axci}{{\bf ci}}
\newcommand{\axcl}{{\bf cl}}

\newcommand{\axca}{{\bf ca}}
\newcommand{\MP}{\textbf{MP}}
\newcommand{\Rn}{\ensuremath{\textbf{R}_\neg}}
\newcommand{\Rb}{\ensuremath{\textbf{R}_\circ}}
\newcommand{\RE}{\ensuremath{\textbf{R}_\square}}
\newcommand{\termvalue}[1] {\lbrack\!\lbrack #1 \rbrack\!\rbrack}

\newcommand{\kax}{\textbf{Ax1}}
\newcommand{\axTrans}{\textbf{Ax2}}
\newcommand{\axed}{\textbf{Ax3}}
\newcommand{\axeea}{\textbf{Ax4}}
\newcommand{\axeeb}{\textbf{Ax5}}
\newcommand{\axouda}{\textbf{Ax6}}
\newcommand{\axoudb}{\textbf{Ax7}}
\newcommand{\axoue}{\textbf{Ax8}}
\newcommand{\axouimp}{\textbf{Ax9}}
\newcommand{\axtnd}{\textbf{Ax10}}
\newcommand{\axexp}{\textbf{bc1}}
\newcommand{\axmod}{\textbf{AxMod}}
\newcommand{\axpem}{\textbf{PEM}}

\newcommand{\axexpl}{\textbf{exp}}
\newcommand{\axce}{{\bf ce}}
\newcommand{\axcf}{{\bf cf}}

\newcommand{\defin}{\ensuremath{~\stackrel{\text{{\tiny def }}}{=}~}}

\newcommand{\imp}{\rightarrow}
\newcommand{\sse}{\leftrightarrow}
\newcommand{\cons}{\ensuremath{{\circ}}}

\newcommand{\wneg}{\ensuremath{\lnot}}
\newcommand{\sneg}{\ensuremath{{\sim}}}

\newtheorem{thm}{\hspace*{10pt}{\sc Theorem}\rm}[section]

\newtheorem{example}[thm]{\hspace*{10pt}{\sc Example}\rm}

\newtheorem{lem}[thm]{\hspace*{10pt}{\sc Lemma}\rm}

\newtheorem{definition}[thm]{\hspace*{10pt}{\sc Definition}\rm}

\newtheorem{cor}[thm]{\hspace*{10pt}{\sc Corollary}\rm}

\newtheorem{prop}[thm]{\hspace*{10pt}{\sc Proposition}\rm}

\newtheorem{remark}[thm]{\hspace*{10pt}{\sc Remark}\rm}

\newenvironment{proof}[1][Proof]{\noindent\textit{#1. } }{\hfill$\square$\vspace{6pt}}

\newenvironment{defn}{\begin{definition}}{\end{definition}}
\newenvironment{teom}{\begin{thm}}{\end{thm}}
\newenvironment{lemma}{\begin{lem}}{\end{lem}}
\newenvironment{coro}{\begin{cor}}{\end{cor}}
\newenvironment{rem}{\begin{remark}}{\end{remark}}
\newenvironment{rems}{\begin{remark}}{\end{remark}}
\newenvironment{exem}{\begin{example}}{\end{example}}

\title{Logics of Formal Inconsistency enriched with replacement: \\
	an algebraic and modal account\footnote{This article has been accepted for publication and will appear in a revised form in The Review of Symbolic Logic, published by Cambridge University Press on behalf of the Association for Symbolic Logic. \\ \copyright Association for Symbolic Logic, 2021}}

\author{Walter Carnielli$^1$, Marcelo E. Coniglio$^{1,2}$ and
David Fuenmayor$^3$\\
\small $^1$Centre for Logic, Epistemology and the History of Science - CLE\\
\small University of Campinas, Brazil\\
\small $^2$Institute of Philosophy and the Humanities - IFCH\\
\small University of Campinas, Brazil\\
\small $^3$Department of Computer Science\\ University of Luxembourg, Luxembourg\\
\small Email: {\tt$\{$walterac, coniglio$\}$@unicamp.br}  \ and \ {\tt david.fuenmayor@uni.lu}
\date{}
}

\begin{document}

\maketitle

\begin{abstract}
It is customary to expect   from a  logical system  that it  can be  \emph{algebraizable}, in the sense that an algebraic companion of the deductive machinery can always be  found.
Since the inception of da Costa's paraconsistent calculi,   algebraic   equivalents  for such systems  have been  sought. It is known, however, that these systems are not self-extensional (i.e., they do not satisfy the {\em replacement property}). More than this, they are  not algebraizable in the sense of Blok-Pigozzi. The same negative results hold for several systems of the hierarchy of paraconsistent logics known as {\em Logics of Formal Inconsistency} (\lfis). Because of this, several systems belonging to this class of logics are only characterizable by semantics of a non-deterministic nature. This paper offers a solution for two open problems
in the domain of paraconsistency, in particular connected to algebraization of \lfis, by extending with rules several \lfis\ weaker than $C_1$, thus obtaining the replacement property (that is, such \lfis\ turn out to be self-extensional). Moreover, these logics become algebraizable in the standard Lindenbaum-Tarski's sense by a suitable variety of Boolean algebras extended with additional operations. The weakest \lfi\ satisfying replacement presented here is called \rmbc, which is obtained from the basic \lfi\ called \mbc. Some axiomatic extensions of \rmbc\ are also studied. In addition, a neighborhood semantics is defined for such systems. It is shown that \rmbc\ can be defined within the minimal bimodal non-normal logic ${\bf E} {\oplus} {\bf E}$  defined by the fusion
of  the non-normal modal logic  {\bf E} with itself. Finally, the framework is extended to first-order languages. \rqmbc, the  quantified extension of \rmbc, is shown to be
sound and  complete w.r.t.~the proposed algebraic semantics.
\end{abstract}

\newpage

\tableofcontents

\

\

\

\section{Introduction: The quest for the  algebraic counterpart of paraconsistency} \label{intro}

It is customary to expect   from a  logical system  that it  can be  \emph{algebraizable}, in the sense that an algebraic companion of the deductive machinery can always be  found. When this happens, several  logical problems can be faithfully and conservatively translated into some given  algebra, and  then algebraic tools can be used to tackle them.  This happens so naturally with the brotherhood between classical  logic and Boolean  algebra, that  a similar relationship  is expected to hold for   non-standard logics as well. And indeed, this  holds for some, but not for all logics.  In any case,  the task of finding such an algebraic counterpart is far from trivial. The intuitive idea  behind the search for  algebraization for a  given logic system, generalizing the  pioneering proposal of Lindenbaum and Tarski, usually starts by trying to find   a congruence on the set of formulas that could be used to produce a quotient algebra, defined over  the algebra of formulas of the logic.
 
One of the favorite methods to set up congruences   is to check the validity of a fundamental  property called \emph{replacement} or (IpE)  (acronym for \emph{intersubstitutivity  by provable equivalents}. Intuitively, (IpE) states that if $\alpha$ and $\beta$ are logically equivalent, then the replacement  in a formula $\gamma$ of any ocurrences of $\alpha$ by $\beta$ produces a formula logically equivalent to $\gamma$.  A logic enjoying replacement is usually called {\em self-extensional}.  

Finding an algebraization for the  logics of the  hierarchy  $C_n$ of paraconsistent logics,  introduced in~\cite{dC63}, constitutes a paradigmatically difficult case. Recall that  a logic with a negation $\neg$ is paraconsistent if it  can incorporate contradictions (w.r.t. $\neg$) that  do not trivialize the system (different from what happens, for instance, with the classical or intuitionistic negation). The idea of da Costa's systems is to define, for each $C_n$, an unary connective $\circ_n$ (where $\circ_n \alpha$ means that $\alpha$ is `well-behaved', or `classically behaved' in  $C_n$) in the following sense: in spite of there existing formulas $\alpha$ and $\beta$ such that  $\beta$ does not follow from the contradictory set  $\{\alpha,\neg\alpha\}$, $\beta$ is always derivable in $C_n$ from $\{\alpha,\neg\alpha,\circ_n\alpha\}$, for every $\alpha$ and $\beta$ (thus, the {\em explosion principle} of negation is only  guaranteed for contradictions involving well-behaved formulas). For instance,  well-behavedness  is defined by $\circ_1\alpha \defin \neg(\alpha \land \neg\alpha)$ in $C_1$, the first logic of da Costa's family.  The approach to paraconsistency of da Costa was afterwards generalized in~\cite{CM} through the notion of {\em Logics of Formal Inconsistency} (\lfis), in which the operator of well-behavedness (rebaptized as {\em consistency operator}, and denoted by $\circ$) is allowed to be  a primitive one.
  
It is known since a long time that  (IpE) does not hold  for  $C_1$ (see~\citet[Corollary to Theorem~1]{dac:guil:64}; a proof for the calculi $C_n$ can be found in~\citet[Theorem~1]{urbas:89}). As a consequence of this failure, a  direct  algebraization (in the sense of  Lindenbaum-Tarski) for this logic is not possible. This blocks  the way for  other, weaker calculi of the  hierarchy $C_n$, since when one logic is algebraizable, so are its extensions. But there are  further possibilities for  algebraization, and the search continued until a proof was presented in~\cite{mort:1980}, establishing  that no  non-trivial quotient algebra is definable for  $C_1$,  or for any logic weaker than  $C_1$. This entails that $C_1$ (together with  its subsystems) is not algebraizable in the general sense of~\cite{blok:pig:89} (a simpler proof of this fact was obtained in~\cite{lew.mik:schw:1991}).  This result was generalized in \citet[Theorem~3.83]{CM} to \cila, the presentation of  $C_1$ in the language of the \lfis\ (where the consistency conective $\circ$ is primitive instead of being defined in terms of the other connectives). We obtain as a consequence that no such algebraization is possible for any other of the \lfis\ weaker than \cila\ to be discussed in Section \ref{Sexte-mbc} (i.e., the systems \mbc, \mbcciw, \bc\ and \ci\ studied in \cite{CM}, \cite{CCM}, and \cite{CC16}). The same reasoning applies to every calculus $C_n$ in the infinite da Costa's hierarchy, since they are weaker than $C_1$.

Some extensions  of $C_1$  having non-trivial quotient algebras have been proposed in the literature. In~\cite{mort:1989}, for instance, it is proposed an infinite number of  intermediate logics  between $C_1$ and classical logic    called $C_{n/(n+1)}$, for $n \geq 1$. In each $C_{n/(n+1)}$ it is possible to define a  non-trivial logical congruence $\sim_n$ by means of a formula $\delta_n(\alpha,\beta)$ such that $\alpha \sim_n \beta$ iff $\vdash_{C_{n/(n+1)}} \delta_n(\alpha,\beta)$ \citep[see][Theorem~4.4]{mort:1989}. However, this feature is strictly weaker than being self-extensional.

Some other  types of algebraic-like counterparts have been investigated, for instance, in \cite{car:alcant:1984} and  \cite{seoa:alcant:1991} a class of structures for the logic  $C_1$  (called {\em da Costa algebras})  was defined,   permitting   a Stone-like representation theorem. In this way,  every da Costa algebra is isomorphic to a paraconsistent algebra of sets, making   $C_1$  closer to traditional  mathematical objects. On the other hand, in~\cite{CT20}  a semantical characterization (which constitutes a decision procedure) is given for  the logics $C_n$ in terms of {\em restricted non-deterministic matrices}.
These structures are non-deterministic matrices (that is, logical matrices
in which the connectives can take several values instead of a single one; cf.~\citet{avr:05}) such that the valuations must satisfy some (decidable) conditions.

To the best of our knowledge, the search for da Costa-like paraconsistent systems satisfying  (IpE) started in the studies by \cite{urbas} and~\cite{sylvan}. The former, a PhD thesis supervised by Sylvan, analyzes (IpE) for some subsystems of the hierarchy of  $J$-systems introduced in~\cite{arruda:dac}. The latter proposes an extension of da Costa's $C_\omega$ by adding the inference rule of contraposition, which guarantees replacement. 
In the realm of \lfis, the situation is more complicated: it can be proved that, for  some subclasses of \lfis,   (IpE) is  unattainable (while preserving paraconsistency), as shown in \citet[Theorem 3.51]{CM} with respect to some axiomatic extensions of the logic \ci, one of the  central systems of the family of \lfis\ which is much  weaker than \cila.

Some  interesting results concerning three-valued self-extensional paraconsistent  logics were obtained in the literature,  in connection with the limitative result mentioned above \citep[namely,][Theorem 3.51]{CM}. In~\cite{avr:bez:17} it was shown that no three-valued paraconsistent logic having an implication can be self-extensional. On the other hand, in~\cite{avr:17} it was shown that there is exactly one self-extensional three-valued paraconsistent logic defined in a signature having conjunction, disjunction and negation. For paraconsistent logics in general, it was shown in~\cite{bez:1998} that no paraconsistent negation $\neg$ satisfying the law of double negation and such that the schema $\neg(\varphi \land \neg\varphi)$ is valid can satisfy (IpE). 
 
Nevertheless, a problem has remained open: to obtain (IpE) for extensions of \ci\  by the addition of weaker forms of contraposition deduction rules while preserving their paraconsistent character, as  discussed in \citet[Subsection 3.7]{CM}. 
The challenge was to find  extensions of \bc\ and \ci\  which  would satisfy (IpE) and still keep their paraconsistent   character. This  paper  meets this challenge in a minimal way: we define the logic \rmbc, an extension by rules of \mbc,  and two 
suitable  extensions of \rmbc, the logics \rbc\ and \rci\ (respectively, extensions of  \bc\ and \ci)  that solve the open problem. Details and exemplary model structures are presented in Examples~\ref{exem-bC} and~\ref{exem-Ci} in Section~\ref{Sexte-mbc}.\footnote{\label{fn:isabelle}We acknowledge the help of model finder \textit{Nitpick} \citep{Nitpick}, integrated in the proof assistant \textit{Isabelle} \citep{nipkow2002isabelle}, in the task of generating suitable models for theories. By finding counter-models, \textit{Nitpick} also spared us the Sisyphean task of trying to prove non-theorems. In order to encode {\rmbc} and its extensions in the higher-order logic of \textit{Isabelle/HOL} we employed the \textit{shallow semantical embedding} approach by \citet{benzmueller2019}. Formalization sources for selected parts of this work are available at: \url{https://github.com/davfuenmayor/LFIs}. See also \citet{Topological_Semantics-AFP} for a recent \textit{Isabelle/HOL} encoding of a generalized topological semantics for paraconsistent and paracomplete logics, which, in a sense, generalizes the neighbourhood semantics presented in Section~\ref{sec:nbhd}.}

The (minimal) logic {\rmbc} and its axiomatic extensions thus form a parallel family of {\lfis} featuring self-extensionality.
We introduce for them, in Section~\ref{sect-rmbc}, a new  kind of  semantic structures: the Boolean  algebras with \lfi\ operators, or BALFIs. Moreover, \rmbc\ is proved to be sound  and complete w.r.t. the class of BALFIs. In Section~\ref{Sexte-mbc} we consider axiomatic extensions of {\rmbc} and give solutions to the open problems mentioned above. Section~\ref{limits} studies some limits for self-extensionality for extensions of {\rmbc} under the condition of paraconsistency.

The paper also  explores some other  directions. Section~\ref{sec:nbhd}  proposes a neighborhood semantics for  \rmbc\ as a special class of BALFIs defined  on powerset Boolean algebras.  Again, \rmbc\ is proved to be sound  and complete w.r.t.~these neighborhood models. Moreover, in Section~\ref{non-normal} it is proved that \rmbc\ can be defined within the minimal bimodal non-normal modal logic \textbf{E}. Such a neighborhood semantics is also  proposed  for axiomatic extensions of \rmbc\  in Section~\ref{kripke-extens}.

The notion of logical consequence used in BALFI semantics for \rmbc, as well as in its neighborhood semantics (as defined  in Section~\ref{sec:nbhd}), is \textit{local} instead of \textit{global} (or \textit{degree-preserving} instead of \textit{truth-preserving} in the sense of~\cite{DegTru}). This requires adapting the usual definition of derivation from premises in a Hilbert calculus. 
In fact, we can, alternatively, consider a global (or truth-preserving) semantics, as it is usually done with algebraic semantics. This will leads us, in Section~\ref{truth-pres}, to  the logic $\rmbc^*$, which is defined by the same Hilbert calculus as for \rmbc, but where  derivations  from premises are defined as in the usual Hilbert calculi.
 
Section~\ref{extension-FOL} is dedicated to extending \rmbc\  to  first-order languages, thus defining the logic \rqmbc, which is proved, in Section~\ref{balfifol} and  Section~\ref{complete}, to be complete w.r.t. BALFI semantics. 

\section{The logic \rmbc} \label{sect-rmbc}

The class of paraconsistent logics known as {\em Logics of Formal Inconsistency} (\lfis, for short) was introduced by \cite{CM}. In their simplest form, they feature a `non-explosive' negation $\neg$, as well as  a (primitive or derived) {\em consistency connective} $\circ$ which allows us to recover the principle of explosion (\textit{ex contradictione sequitur quodlibet}) in a controlled way. 

\begin{defn} \label{defLFI}
Let  ${\bf L}=\langle  \Theta,\vdash \rangle$ be a Tarskian, finitary and structural logic defined over a propositional signature $\Theta$, which contains a
negation $\neg$, and let  $\circ$ be a (primitive or derived) unary connective. Then,
${\bf L}$ is said to be  a {\em Logic of Formal Inconsistency} (\lfi) with respect to $\neg$ and $\circ$ if the following holds:

$\begin{array}{ll}
(i) & \varphi,\neg\varphi\nvdash\psi \ \mbox{ for some $\varphi$ and $\psi$};\\
(ii) & {\circ}\varphi,\varphi,\neg\varphi\vdash\psi \ \mbox{  for every $\varphi$ and $\psi$; } \\
(iii) & \mbox{there are two formulas $\alpha$ and $\beta$ such that}\\
& \hspace{4mm} (iii.a) \ {\circ}\alpha,\alpha \nvdash \beta;\\
& \hspace{4mm} (iii.b) \ {\circ}\alpha, \neg \alpha \nvdash \beta.

\end{array}$
\end{defn}

Condition (i) above signals the non-validity of the classical \textit{principle of explosion}. Condition (ii), also called \textit{principle of gentle explosion}, characterizes {\lfis} in particular. Condition (iii) is required in order to satisfy condition~(ii) in  a non-trivial way. The hierarchy of \lfis\ studied in~\cite{CCM} and~\cite{CC16} starts from a logic called \mbc, which extends  positive classical logic \cplp\ by adding a negation $\neg$ and a primitive {\em consistency} operator $\circ$ satisfying minimal requirements in order to define an \lfi. 

\begin{defn} \label{signa}
In what follows, the following signatures will be considered:
\begin{enumerate}
\item[] $\Sigma_+=\{\land, \lor, \to\}$;
\item[] $\Sigma_{BA}=\{\land, \lor, \to,\bar 0,\bar 1\}$;
\item[] $\Sigma=\{\land, \lor, \to, \wneg, \cons\}$;
\item[] $\Sigma_C=\{\land, \lor, \to,\neg\}$;
\item[] $\Sigma_{C_0}=\{\land, \lor, \to,\neg,\bar 0\}$;
\item[] $\Sigma_{C_e}=\{\land, \lor, \to, \wneg,\bar 0,\bar 1\}$;
\item[] $\Sigma_e=\{\land, \lor, \to, \wneg, \cons,\bar 0,\bar 1\}$; 
\item[] $\Sigma_m=\{\land, \lor, \to, \sneg, \square,\lozenge\}$; and
\item[] $\Sigma_{bm}=\{\land, \lor, \to, \sneg, \square_1,\lozenge_1,\square_2,\lozenge_2\}$.
\end{enumerate}
\end{defn}

If $\Theta$ is a propositional signature, then $For(\Theta)$ will denote the (absolutely free) algebra of formulas over $\Theta$ generated by a given denumerable set $\mathcal{V}=\{p_n \ : \  n \in \mathbb{N}\}$ of propositional variables. Throughout this paper, logic systems will be presented by means of Hilbert calculi (in some cases with two possible definitions of derivations from premises: local and global), and a semantics will be given for them by means of algebraic structures, proving the soundness and completeness of each logic w.r.t. its semantics. For simplicity, any logic will be identified with its Hilbert calculus presentation (w.r.t.~a given notion of derivation in that calculus).

\begin{defn} [Classical Positive Logic] \label{hilCPLP} The {\em classical  positive logic} \cplp\ is defined over the language $For(\Sigma_+)$ by the following Hilbert  calculus:\footnote{It is hard to pinpoint in the history of logic  who first proposed an axiomatization for the full  positive fragment of classical	propositional calculus.  H. B. Curry in  a  book   written   in 1959-1961 for a course (published as~\cite{curry}) deals  in his chapter~5 with   the   positive fragment   of  the  propositional  calculus (i.e.,  with  conjunction, disjunction  and implication, but  without negation). Before that,  \cite{luka}  and~\cite{henkin}, among other authors, studied proper fragments  of positive logic.
	}  \\[2mm]
  {\bf Axiom schemas:}
  \begin{gather}
    \alpha \imp \big(\beta \imp \alpha\big)             \tag{\kax} \\
   \Big(\alpha\imp\big(\beta\imp\gamma\big)\Big) \imp
            \Big(\big(\alpha\imp\beta\big)\imp\big(\alpha\imp\gamma\big)\Big)
              \tag{\axTrans}\\
    \alpha \imp \Big(\beta \imp \big(\alpha \land \beta\big)
                                             \Big)  \tag{\axed}\\
    \big(\alpha \land \beta\big) \imp \alpha         \tag{\axeea}\\
    \big(\alpha \land \beta\big) \imp \beta          \tag{\axeeb}\\
    \alpha \imp \big(\alpha \lor \beta\big)          \tag{\axouda}\\
    \beta \imp \big(\alpha \lor \beta\big)           \tag{\axoudb}\\
    \Big(\alpha \imp \gamma\Big) \imp \Big(
        (\beta \imp \gamma) \imp
           \big(
          (\alpha \lor \beta) \imp \gamma
           \big)\Big)                               \tag{\axoue} \\
    \big(\alpha \imp \beta\big) \lor \alpha          \tag{\axouimp}
  \end{gather}

{\bf Inference rule:}
\[\frac{\alpha \ \ \ \ \alpha\imp
      \beta}{\beta}  \tag{\MP}\]
\end{defn}
 
\

\begin{defn} The logic \mbc, defined over signature $\Sigma$, is obtained from  \cplp\ by adding the following axiom schemas:
  \begin{gather}
    \alpha \lor \lnot \alpha                        \tag{\axtnd}\\
    \cons \alpha \imp \Big(\alpha \imp \big(\lnot \alpha \imp \beta\big)\Big)
                                                     \tag{\axexp}
  \end{gather}
\end{defn}

The logic \mbc\ is an \lfi. Indeed, it is the minimal \lfi\ extending \cplp.

Take into consideration the {\em replacement property}, namely: If $\alpha \sse \beta$ is a theorem then $\gamma[p/\alpha] \sse \gamma[p/\beta]$ is a theorem, for every formula $\gamma(p)$.\footnote{As usual, $\alpha \sse \beta$ is an abbreviation of the formula $(\alpha \to \beta) \land(\beta \to \alpha)$, and $\gamma[p/\alpha]$ denotes the formula obtained from $\gamma$ by replacing every occurrence of the variable $p$ by the formula $\alpha$.}
It is well known that \mbc\ does not satisfy the replacement property  in general. However, it is easy to see that replacement holds in \mbc\ for every formula $\gamma(p)$ over the signature $\Sigma_+$ of \cplp. We  then  introduce  the logic \rmbc\ which extends \mbc\ by adding replacement for every formula over $\Sigma$. From the previous observation, it is enough to add replacement for $\neg$ and $\cons$ as new inference rules. Namely: if $\alpha \sse \beta$ is a theorem then  $\neg\alpha \sse \neg\beta$ (is a theorem), and  if $\alpha \sse \beta$ is a theorem then  $\cons\alpha \sse \cons\beta$ (is a theorem).

Observe, however, that replacement is in fact a metaproperty (since it states that some formula is a theorem from previous formulas which are assumed to be theorems). It is clear that the two inference rules proposed above for inducing replacement are global instead of local (see Section~\ref{truth-pres} below): in order to apply each rule, the corresponding premise must be a theorem. This is an analogous situation to the {\em necessitation rule}   in modal logics. Assuming inference rules of this kind   requires changing the definition of derivation from premises in the resulting Hilbert calculus, as we shall see below.

\begin{defn} \label{rules} The logic \rmbc, defined over signature $\Sigma$, is obtained from  \mbc\ by adding the following inference rules:
\[\frac{\alpha \sse \beta}{\neg\alpha \sse \neg\beta}  \hspace{1cm}(\Rn) \hspace{4cm}\frac{\alpha \sse \beta}{\cons\alpha \sse \cons\beta}  \hspace{1cm}(\Rb)\]
\end{defn}

\begin{defn} [Derivations in \rmbc] \label{deriv-RmbC} \ \\
(1) A {\em derivation} of a formula
$\varphi$ in \rmbc\ is a finite sequence of formulas
$\varphi_1\ldots \varphi_n$ such that $\varphi_n$ is $\varphi$ and, for every
$1 \leq i \leq n$, either  $\varphi_i$ is an instance of an axiom schema of \rmbc, or  $\varphi_i$ is the consequence of some inference rule of \rmbc\  whose premises appear in the sequence $\varphi_1\ldots \varphi_{i-1}$.\\
(2) We say that a formula $\varphi$ is {\em derivable} in \rmbc, or that $\varphi$ is a {\em theorem of \rmbc}, denoted by $\vdash_{\rmbc}\varphi$, if there exists a derivation of $\varphi$ in \rmbc.\\
(3)  Let $\Gamma \cup\{\varphi\}$ be a set of formulas over $\Sigma$. We say that {\em $\varphi$ is derivable in \rmbc\ from the premises $\Gamma$}, and we write $\Gamma\vdash_{\rmbc}\varphi$, if either  $\varphi$ is derivable in \rmbc, or there exists a finite, non-empty subset $\{\gamma_1,\ldots,\gamma_n\}$ of $\Gamma$ such that the formula $(\gamma_1 \wedge(\gamma_2 \wedge(\ldots\wedge(\gamma_{n-1} \wedge \gamma_n)\ldots))) \to \varphi$ is derivable in \rmbc.
\end{defn}

\begin{rems} \label{wojci} \ \\
(1) The order in which we take the conjunction of elements in the set $\{\gamma_1,\ldots,\gamma_n\}$ in item (3) above is immaterial because of the replacement property. In the sequel, we will write $(\gamma_1 \wedge \ldots \wedge \gamma_n)$ instead of $(\gamma_1 \wedge(\gamma_2 \wedge(\ldots\wedge(\gamma_{n-1} \wedge \gamma_n)\ldots)))$.\\
(2) From the previous definition, it follows that $\emptyset \vdash_{\rmbc}\varphi$ \ iff \ $\vdash_{\rmbc}\varphi$.\\
(3) Recall that a consequence relation $\vdash$ is said to be Tarskian and finitary if it satisfies the following properties: (i)~$\Gamma \vdash \alpha$ whenever $\alpha \in \Gamma$; (ii)~if $\Gamma \vdash \alpha$ and $\Gamma \subseteq \Delta$ then $\Delta \vdash \alpha$; (iii)~if $\Gamma \vdash \Delta$ and $\Delta \vdash \alpha$ then $\Gamma \vdash \alpha$, where $\Gamma \vdash \Delta$ means that $\Gamma \vdash \delta$ for every $\delta \in \Delta$; and (iv)~$\Gamma \vdash \alpha$ implies that $\Gamma_0 \vdash \alpha$ for some finite $\Gamma_0$ contained in $\Gamma$. It is easily seen that the consequence relation $\vdash_{\rmbc}$ given in Definition~\ref{deriv-RmbC}(2) is Tarskian and finitary, by virtue of the properties of $\to$ and $\wedge$ in \rmbc\ inherited from  \cplp.
\end{rems}

\begin{teom} \label{IPE-rmbc} Let {\bf L} be the logic \rmbc\ or any axiomatic extension of it over the signature $\Sigma$, in which the derivation from premises is defined as in \rmbc. Then, {\bf L} satisfies the replacement property.
\end{teom}
\begin{proof}
Using the properties of \cplp, we have that 
$$\big((\alpha \sse \alpha') \land (\beta \sse \beta')\big) \to \big((\alpha \,\#\, \beta) \sse (\alpha' \,\#\, \beta')\big)$$ 
is a theorem of {\bf L}, for $\# \in \{\land, \vee, \to\}$. From this, $(\alpha \,\#\, \beta) \sse (\alpha' \,\#\, \beta')$ is a theorem of {\bf L} provided that $(\alpha \sse \alpha')$ and $(\beta \sse \beta')$ are theorems of {\bf L}. On the other hand, the rules $(\Rn)$ and $(\Rb)$ guarantee that $(\#\,\alpha \sse \#\,\alpha')$ is a theorem of {\bf L} provided that $(\alpha \sse \alpha')$ is a theorem of {\bf L}, for $\# \in \{\neg, \cons\}$. From this, it is straightforward to prove, by induction on the complexity of the formula $\gamma(p)$,\footnote{By convenience, and as is usually done in \lfis\ (see, e.g., \cite{CC16}), the complexity of $\cons\gamma$ is taken to be stricty greater than the complexity of $\neg\gamma$.} that {\bf L} satisfies replacement: If $\alpha \sse \beta$ is a theorem of {\bf L} then $\gamma[p/\alpha] \sse \gamma[p/\beta]$ is a theorem of {\bf L}. The details are left to the reader.
\end{proof}

By considering once again  the properties of $\wedge$ and $\to$ inherited from \cplp, and by the notion of derivation in \rmbc, it is easy to see that the {\em Deduction Metatheorem} holds in \rmbc:

\begin{teom}[Deduction Metatheorem for \rmbc] \ \\
$\Gamma, \varphi \vdash_{\rmbc}  \psi$ \ if and only if $\Gamma \vdash_{\rmbc} \varphi \to \psi$.
\end{teom}

An algebraic semantics for \rmbc\ will be given by means of a suitable class of Boolean algebras extended with additional unary operations, which interpret the non-classical connectives. We term these additional operations: \textit{\lfi~operators}. In view of the definition  of derivations in \rmbc\  discussed above,  the semantic consequence relation will be {\em degree preserving} instead of {\em truth preserving} \citep[see]{DegTru}. By analogy to modal logics, we may call such a semantics \textit{local} instead of \textit{global}. We will return to this point in Section~\ref{truth-pres}. 

\begin{defn} [BALFIs] \label{defBALFI}
A {\em Boolean algebra with \lfi\ operators} (BALFI, for short) is an algebra
$\B = \langle A, \wedge, \vee, \to,\neg,\cons,0,1 \rangle$ over $\Sigma_e$ such that its reduct $\mathcal{A}=\langle A, \wedge, \vee, \to,0,1 \rangle$ to $\Sigma_{BA}$ is a Boolean algebra and the unary operators $\neg$ and $\cons$ satisfy: $a \vee \neg a = 1$ and $a \wedge \neg a \wedge \cons a= 0$, for every $a \in A$. The variety\footnote{Recall that a class of algebras is a \textit{variety} if it can be axiomatized by means of equations.} of BALFIs will be denoted by $\mathbb{BI}$.
\end{defn}

\begin{rems} \label{obsBALFIs} \ \\  
(1) The set $A$ is called the {\em carrier} or {\em domain} of \B\ (and also of \A). \\
(2) Recall that an implicative lattice is an algebra $\mathcal{A}=\langle A,\land,\lor,\to,1\rangle$ where the reduct $\langle A,\land,\lor,1\rangle$ is a lattice with top element 1 in which $\bigvee \{ c \in A \ : \  a \wedge c \leq b\}$ exists for every $a,b \in A$, and $a \to b \defin \bigvee \{ c \in A \ : \  a \wedge c \leq b\}$ for every $a,b \in A$.\footnote{Here $\leq$ is the partial order associated to the lattice, namely: $a \leq b$ iff $a=a\wedge b$ iff $b=a \vee b$, and $\bigvee X$ denotes the supremum of the set $X \subseteq A$ w.r.t. $\leq$, provided that it exists. Observe that $a \leq b$ iff $a \to b =1$.} Implicative lattices are in fact distributive. A Heyting algebra is an implicative lattice with least element $0$, where the pseudo-complement is defined as $\sneg a \defin a \to 0$ \citep[see, e.g.,][11.2 and~11.3]{dunn}. A Boolean algebra is a Heyting algebra such that $a \vee \sneg a =1$ for every $a$.	This justifies the use of the signature $\Sigma_{BA}$ for presenting Boolean algebras. \\
(3) BALFIs are in a sense analogous to the {\em Boolean algebras with operators} (BAOs) used as semantics for modal logics (see, e.g., \cite{jons:93}). Observe, however, that \lfi~operators do not generally distribute over meets or joins. \\
\end{rems} 

\begin{defn} [Degree-preserving BALFI semantics] \ \\ \label{degBALFI}
(1) A {\em valuation} over a BALFI \B\ is a homomorphism $v:For(\Sigma) \to \B$ of algebras  over $\Sigma$ (where \B\ is considered, in particular,  as an algebra over $\Sigma$).\\
(2) Let $\varphi$ be a formula in $For(\Sigma)$. We say that $\varphi$ is {\em valid in $\mathbb{BI}$}, denoted by $\models_{\mathbb{BI}} \varphi$, if, for every BALFI \B\ and every valuation $v$ over it,  $v(\varphi)=1$.\\
(3) Let $\Gamma \cup\{\varphi\}$ be a set of formulas in $For(\Sigma)$. We say that {\em $\varphi$ is a local} (or {\em degree-preserving}) consequence of $\Gamma$ in $\mathbb{BI}$, denoted by $\Gamma \models_{\mathbb{BI}} \varphi$, if either  $\varphi$ is valid in $\mathbb{BI}$, or there exists a finite, non-empty subset $\{\gamma_1,\ldots,\gamma_n\}$ of $\Gamma$ such that, for every BALFI \B\ and every valuation $v$ over it, $\bigwedge_{i=1}^nv(\gamma_i) \leq  v(\varphi)$ (recalling that,  in any Boolean algebra \A, the order is given by: $a \leq b$ iff $a \to b=1$).
\end{defn}

\begin{rem}  \ \\  \label{obsded}
(1) Defining BALFI valuations as homomorphisms in this way means that $v(\#\varphi) = \# v(\varphi)$, \text{for} $\#\in \{\neg, \cons\}$, and $v(\varphi \,\#\, \psi) = v(\varphi) \,\#\, v(\psi)$, \text{for} $\#\in \{\wedge,\vee, \to\}$.\\
(2) Note that $\Gamma \models_{\mathbb{BI}} \varphi$ iff either $\varphi$ is valid in $\mathbb{BI}$, or there exists a finite, non-empty subset $\{\gamma_1,\ldots,\gamma_n\}$ of $\Gamma$ such that
$(\gamma_1 \wedge \ldots \wedge \gamma_n) \to \varphi$
is valid. This follows easily from the definitions.
\end{rem}

\begin{teom} [Soundness of \rmbc\ w.r.t. $\mathbb{BI}$] \ \\ \label{sound-BALFI} Let $\Gamma \cup \{\varphi\} \subseteq For(\Sigma)$. Then: $\Gamma \vdash_{\rmbc} \varphi$ \ implies that \ $\Gamma \models_{\mathbb{BI}} \varphi$.
\end{teom}
\begin{proof}
Let $\varphi$ be a an instance of an axiom of \rmbc. It is immediate to see that, for every \B\ and every valuation $v$ on it, $v(\varphi)=1$. Now, let $\alpha,\beta \in For(\Sigma)$.  If $v(\alpha \to \beta)=1$ and $v(\alpha)=1$ then, since $v(\alpha \to \beta)=v(\alpha) \to v(\beta)$, it follows that $v(\beta)=1$. On the other hand, if $v(\alpha \sse \beta)=1$ then $v(\alpha)=v(\beta)$ and so  $v(\#\alpha)=\#v(\alpha)=\#v(\beta)= v(\#\beta)$, hence  it follows that $v(\#\alpha \sse \#\beta)=1$ for $\# \in \{\neg,\cons\}$. From this, by induction on the length of a derivation of $\varphi$ in \rmbc, it can be easily proven that $\varphi$ is valid in $\mathbb{BI}$ whenever $\varphi$ is derivable in \rmbc. Now, suppose that  $\Gamma \vdash_{\rmbc} \varphi$. If $\vdash_{\rmbc} \varphi$ then, by the observation above, $\varphi$ is valid in $\mathbb{BI}$ and so   $\Gamma \models_{\mathbb{BI}} \varphi$. On the other hand, if there exists a finite, non-empty subset $\{\gamma_1,\ldots,\gamma_n\}$ of $\Gamma$ such that 
$\vdash_{\rmbc} (\gamma_1 \wedge \ldots \wedge \gamma_n) \to \varphi$
then, once again by the observation above,
$\models_{\mathbb{BI}} (\gamma_1 \wedge \ldots \wedge \gamma_n) \to \varphi$.
This shows that $\Gamma \models_{\mathbb{BI}} \varphi$, by Remark~\ref{obsded}.
\end{proof}

\begin{teom} [Completeness of \rmbc\ w.r.t. $\mathbb{BI}$] \ \\ \label{comple-BALFI} Let $\Gamma \cup \{\varphi\} \subseteq For(\Sigma)$. Then: $\Gamma  \models_{\mathbb{BI}} \varphi$ \ implies that \ $\Gamma \vdash_{\rmbc} \varphi$.
\end{teom}
\begin{proof}
Define the following relation on $For(\Sigma)$: $\alpha \equiv \beta$ iff $\vdash_{\rmbc} \alpha \sse \beta$. It is clearly an equivalence relation, by the properties of \cplp. Let $A_{can}\defin For(\Sigma)/_{\equiv}$ be the quotient set, and define over $A_{can}$ the following operations: $[\alpha] \,\#\, [\beta] \defin [\alpha \# \beta]$, for $\# \in \{\land,\lor,\imp\}$, where $[\alpha]$ denotes the equivalence class of $\alpha$ w.r.t. $\equiv$. Let $0 \defin [\alpha \wedge \neg \alpha \wedge \cons \alpha]$ and $1 \defin [\alpha \vee \neg \alpha]$. These operations and constants are clearly well-defined, by Theorem~\ref{IPE-rmbc}, and so they induce a structure of Boolean algebra over the set $A_{can}$, which will be denoted by $\mathcal{A}_{can}$. Let $\mathcal{B}_{can}$ be its expansion to $\Sigma_e$ by defining $\#[\alpha] \defin [\#\alpha]$, for $\# \in \{\neg,\cons\}$. By Theorem~\ref{IPE-rmbc} these operations are well-defined, and it is immediate to see that $\mathcal{B}_{can}$ is a BALFI. Let $v_{can}:For(\Sigma) \to \mathcal{B}_{can}$ given by $v_{can}(\alpha)=[\alpha]$. Clearly $v_{can}$ is a valuation over  $\mathcal{B}_{can}$ such that $v_{can}(\alpha)=1$ iff $\vdash_{\rmbc} \alpha$. 

Now, suppose that $\Gamma  \models_{\mathbb{BI}} \varphi$, and recall Remark~\ref{obsded}. If $\models_{\mathbb{BI}} \varphi$ then, in particular, $v_{can}(\varphi)=1$ and so $\vdash_{\rmbc} \varphi$. From this, $\Gamma \vdash_{\rmbc} \varphi$. On the other hand, if there exists a finite, non-empty subset $\{\gamma_1,\ldots,\gamma_n\}$ of $\Gamma$ such that $\models_{\mathbb{BI}} (\gamma_1 \wedge \ldots \wedge \gamma_n) \to \varphi$ then, in particular, $v_{can}((\gamma_1 \wedge \ldots \wedge \gamma_n) \to \varphi)=1$. This means that $\vdash_{\rmbc} (\gamma_1 \wedge \ldots \wedge \gamma_n) \to \varphi$ and so  $\Gamma \vdash_{\rmbc} \varphi$.
\end{proof}

\begin{defn} \label{modcan}
The pair $\langle \mathcal{B}_{can},v_{can}\rangle$ defined in the proof of Theorem~\ref{comple-BALFI} is called {\em the canonical model} of \rmbc.
\end{defn}

\begin{exem} [BALFIs over $\wp(\{w_1,w_2\})$] \label{allBALFIs} Let $\A_4=\wp(\{w_1,w_2\}) = \{0, a,b,1\}$ be the powerset of $W_2=\{w_1,w_2\}$ such that $0=\emptyset$, $a=\{w_1\}$, $b=\{w_2\}$ and $1=\{w_1,w_2\}$. Then, the BALFIs defined by expanding the Boolean algebra $\A_4$ are shown in Figure~\ref{fig:BALFI-models1}.
\begin{center}
\begin{figure}[h!] 
  \centering
    \includegraphics[width=0.5\textwidth]{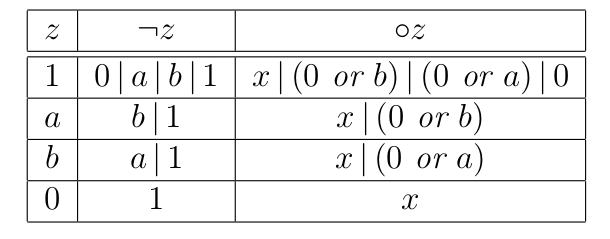} 
    \caption{Set of BALFIs over $\A_4$ (the symbol $|$ separates the possible options for the values of $\neg z$ and $\cons z$ for every value of $z$, while $x$ stands for any element of $\A_4$)} \label{fig:BALFI-models1}
\end{figure}
\end{center}%
\noindent On each row of  Figure~\ref{fig:BALFI-models1}, each choice in the ith position of the sequence of options in the column for $\neg z$ forces a choice  in the ith position of the sequence of options in the column for $\cons z$. For instance, if in the current BALFI we choose $\neg 1=b$ then there are two possibilities for the value of $\cons 1$ in that BALFI: either $\cons 1=0$ or $\cons 1=a$. On the other hand, by choosing that $\neg a=1$ it forces that either $\cons a=0$ or $\cons a=b$. Otherwise, if $\neg a=b$ then $\cons a$ can be arbitrarily chosen.  
\end{exem}

\begin{remark} \label{obsRmbC}
Observe that the rules (\Rn) and (\Rb) do not ensure that $ \vdash_\rmbc (\alpha \sse \beta) \to (\neg\alpha \sse \neg\beta)$ and  $ \vdash_\rmbc (\alpha \sse \beta) \to (\cons\alpha \sse \cons\beta)$ in general. Consider, for instance $\alpha=p$ and $\beta=q$ where $p$ and $q$ are two different propositional variables, and take the BALFI \B\ (in Figure~\ref{fig:BALFI-models2}) defined over the Boolean algebra $\wp(\{w_1,w_2\})$, according to Example~\ref{allBALFIs}.

\begin{center}
\begin{figure}[h!]
\vspace*{-0.45cm}
  \centering
    \includegraphics[width=0.2\textwidth]{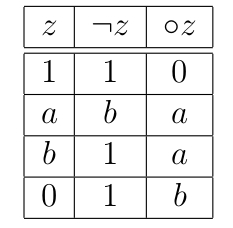}
    \caption{Counter-model BALFI (over $\A_4$) for $\vdash_\rmbc (\alpha \sse \beta) \to (\#\alpha \sse \#\beta)$; $\#\in\{\neg, \circ\}$}\label{fig:BALFI-models2}
\end{figure}
\end{center}

\noindent Now, consider a valuation $v$ over \B\  such that $v(p)=a$ and $v(q)=1$. Hence $v(\neg p)=b$, $v(\neg q)=1$, $v(\cons p)=a$ and $v(\cons q) = 0$. From this $v(p \sse q)=a$ and  $v(\neg p \sse \neg q)= v(\cons p \sse \cons q)= b$. Therefore $v((p \sse q) \to (\neg p \sse \neg q))=v((p \sse q) \to (\cons p \sse \cons q)) =b$. That is, none of the last two formulas is valid in \rmbc. Of course both formulas hold if  $\vdash_\rmbc (\alpha \sse \beta)$, by (\Rn) and (\Rb).

Evidently, \rmbc\ is paraconsistent: in the BALFI \B\ we just defined  above, the given valuation $v$ shows that $q,\neg q \nvdash_{\rmbc} p$. Now, recalling Definition~\ref{defLFI}, \rmbc\ is clearly an \lfi. Consider the following BALFI $\B'$ (in Figure~\ref{fig:BALFI-models3}) defined over $\wp(\{w_1,w_2\})$, using again Example~\ref{allBALFIs}.

\begin{center}
\begin{figure}[h!]
\vspace*{-0.25cm}
  \centering
    \includegraphics[width=0.2\textwidth]{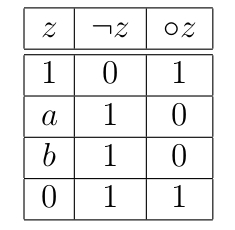}
    \caption{BALFI (over $\A_4$) validating the \textit{principle of gentle explosion} in a non-trivial way}\label{fig:BALFI-models3}
\end{figure}
\end{center}
\noindent Take a valuation $v'$ over $\B'$ such that $v'(p)=1$ and $v'(q)=a$. This shows that $p,\cons p \nvdash_{\rmbc} q$.  Now, a valuation $v''$  over $\B'$ such that $v''(p)=0$ and $v''(q)=b$ shows that $\neg p, \cons p \nvdash_{\rmbc} q$. On the other hand, from axiom (\axexp), the axioms of \cplp, and by Definition~\ref{rules}(3),  it is the case that $\alpha, \neg\alpha, \cons\alpha \vdash_{\rmbc} \beta$ for any formulas $\alpha$ and $\beta$.
\end{remark}

\section{Adding replacement to extensions of \mbc: A solution to an open problem} \label{Sexte-mbc}

In~\cite{CM}, the first study on \lfis, the replacement property was analyzed under the name (IpE), presented in the following (equivalent) way:\\

(IpE) \ if $\alpha_i \dashv\vdash \beta_i$ (for $1 \leq i \leq n$) then $\varphi(\alpha_1,\ldots,\alpha_n) \dashv\vdash \varphi(\beta_1,\ldots,\beta_n)$\\[2mm]

\noindent for any formulas $\alpha_i, \beta_i, \varphi$. In that paper an important question was posed: to define extensions of \bc\ and \ci\ (two axiomatic extensions of \mbc\ to be analyzed below) which satisfy (IpE) while still being paraconsistent.\footnote{In~\cite{CM}, page~41, one reads: ``The question then would be if (IpE) could be obtained for {\em real} \lfis''.
On page~54, after observing that in extensions of \ci\ it is enough ensuring (IpE) for $\wneg$, since it implies (IpE) for $\cons$,  it is said that ``We suspect that this can be done, but we shall leave it as an open problem at this point''. Finally, they observe on page~55,  footnote~17 that certain 8-valued matrices presented by Urbas satisfy (IpE) for  an extension of \bc. However, this logic is not paraconsistent. After this, they claim that ``the question is still left open as to whether there are {\em paraconsistent} such extensions of \bc !''.} In this section, we will present a general (in the sense of minimal) solution to that open problem, obtained by extending axiomatically \rmbc. In what follows, some \lfis\ which are axiomatic extensions of \mbc\ (\bc\ and \ci, among others) will be enriched with replacement, in a similar way as we did for \rmbc.
\begin{defn} [Some extensions of \mbc]  \label{exte-mbc}
Consider the following axioms presented in~\cite{CM}  and~\cite{CC16}:
  \begin{gather}
    \cons\alpha \lor (\alpha \land \wneg \alpha)             \tag{\axwci}\\
    \wneg\cons\alpha \imp (\alpha \land \wneg \alpha)             \tag{\axci}\\
\wneg(\alpha \land \wneg \alpha) \imp \cons\alpha              \tag{\axcl}\\
\wneg\wneg\alpha \imp \alpha             \tag{\axcf}\\
    \alpha \imp \wneg\wneg\alpha             \tag{\axce}\\
   (\cons\alpha \land \cons\beta) \imp \cons(\alpha\land \beta)  \tag{\axca$_\land$}\\
   (\cons\alpha \land \cons\beta) \imp \cons(\alpha\lor \beta)  \tag{\axca$_\lor$}\\
   (\cons\alpha \land \cons\beta) \imp \cons(\alpha\imp \beta)  \tag{\axca$_\imp$}
  \end{gather}
\end{defn}

\begin{remark} \label{obsRmbC-ext} \ \\ 
(1) Axiom (\axwci) was introduced by \citet{avr:05} by means of two axioms, ({\bf k1}): $\cons\alpha \lor \alpha$ \ and  \ ({\bf k2}): $\cons\alpha \lor  \wneg \alpha$. Strictly speaking, ({\bf k1}) and ({\bf k2}) were presented as rules in a standard Gentzen calculus.\\
(2) A classical negation (denoted by $\sneg$) is definable in \mbc\ as $\sneg\alpha = \alpha \to \bot$. Here, $\bot$ denotes any formula of the form $\beta \wedge \neg \beta \wedge \cons\beta$. It is not difficult to show that $\cons\alpha \imp \sneg(\alpha \land \wneg \alpha)$ is a theorem of \mbc, and thus of \rmbc. Also notice that the converse of this formula, namely, $\sneg(\alpha \land \wneg \alpha) \imp \cons\alpha$, is derived in \mbc\ (hence, in \rmbc) from (\axwci).\\
(3) It can also be shown that $(\alpha \land \wneg \alpha) \imp \wneg\cons\alpha$, and $\cons\alpha \imp \wneg(\alpha \land \wneg \alpha)$ are theorems of \mbc~and therefore of \rmbc~(for instance, by applying the Deduction Metatheorem to~Theorem~49~(ii) and~(iii) in~\cite{CCM}). Notice that these formulas are the `converses' of (\axci) and (\axcl) respectively.
\end{remark}

\begin{defn} Let $\B = \langle A, \wedge, \vee, \to,\neg,\cons,0,1 \rangle$ be a BALFI, and let $\varphi$ be a formula over $\Sigma$. We say that $\B$ {\em is a model of $\varphi$} (considered as an axiom schema), denoted by $\B\Vdash \varphi$, if $v(\varphi)=1$ for every valuation $v$ over $\B$.
\end{defn}

As  mentioned in Remark~\ref{obsBALFIs} the Boolean complement in Boolean algebras defined over the signature $\Sigma_{BA}$ is given by $\sneg a = a \to 0$. 

\begin{prop} \label{models-axioms}
Let $\B = \langle A, \wedge, \vee, \to,\neg,\cons,0,1 \rangle$ be a BALFI. Then:\\[1mm]
(1) $\B$ is a model of $(\axwci)$ iff $\B$ satisfies the equation $\cons a = \sneg(a \wedge \neg a)$ for every $a \in A$;\\[1mm]
(2) $\B$ is a model of $(\axci)$ iff $\B$ satisfies the equation $\neg\cons a = a \wedge \neg a$ for every $a \in A$;\\[1mm]
(3) $\B$ is a model of $(\axcl)$ iff $\B$ satisfies the equation $\cons a = \neg(a \wedge \neg a)$ for every $a \in A$;\\[1mm]
(4) $\B$ is a model of $(\axcf)$ iff $\B$ satisfies the equation $a \wedge \neg\neg a= \neg\neg a$ for every $a \in A$;\\[1mm]
(5) $\B$ is a model of $(\axce)$ iff $\B$ satisfies the equation $a \wedge \neg\neg a= a$ for every $a \in A$;\\[1mm]
(6) $\B$ is a model of $(\axca_\#)$ iff $\B$ satisfies the equation $(\cons a \wedge \cons b) \wedge \cons(a \# b)= \cons a \wedge \cons b$ for every $a,b \in A$, for each $\# \in \{\wedge,\vee, \to\}$.
\end{prop}

\begin{proof}
We prove~(1). Suppose that $\B$ is a model of $(\axwci)$. Then, $v(\cons p \lor (p \land \wneg p))=\cons v(p) \lor (v(p) \land \wneg v(p))=1$ for every variable $p$ and any valuation $v$. Since $v(p)$ can take  an arbitrary value in $A$, this is equivalent to say that $\cons a \lor (a \land \wneg a)=1$ for every $a \in A$. Since $\cons a \land (a \land \wneg a)=0$ (by definition of BALFIs), it follows that $\cons a= \sneg (a \land \wneg a)$. The converse is immediate.
Proofs for (2) and (3) are obtained in an analogous way (recalling Remark~\ref{obsRmbC-ext}~(3)).
The proofs of (4), (5), and (6) are immediate from the definitions.
\end{proof}

Let $Ax$ be a set formed by one or more of the axiom  schemas introduced in Definition~\ref{exte-mbc}, and let $\mbc(Ax)$ be the logic defined by the Hilbert calculus obtained from \mbc\ by adding the set $Ax$ of axiom  schemas. Let $\mathbb{BI}(Ax)$ be the class of BALFIs which are models of every axiom in $Ax$. Clearly, $\mathbb{BI}(Ax)$ is a variety of algebras, since it is characterized by a set of equations. Finally, let $\rmbc(Ax)$ be the logic obtained from  \rmbc\ by adding the set $Ax$ of axiom  schemas. It is simple to adapt the proofs of Theorems~\ref{sound-BALFI} and~\ref{comple-BALFI} (in particular, by defining for each logic the corresponding canonical model, as in Definition~\ref{modcan}) in order to obtain the following:

\begin{teom}[Soundness and completeness of $\rmbc(Ax)$ w.r.t. $\mathbb{BI}(Ax)$]  \label{adeq-ext-BALFI} Let $\Gamma \cup \{\varphi\} \subseteq For(\Sigma)$. Then: $\Gamma \vdash_{\rmbc(Ax)} \varphi$ \ \ if and only if \ \ $\Gamma \models_{\mathbb{BI}(Ax)} \varphi$.
\end{teom}

From this important result, some properties of the calculi  $\rmbc(Ax)$ can be easily proven by algebraic methods, that is, by means of BALFIs. For instance:

\begin{prop} \label{collapse} $\mathbb{BI}(\{\axci,\axcf\})=\mathbb{BI}(\{\axcl,\axcf\})=\mathbb{BI}(\{\axci, \axcl,\axcf\})$. Hence, the logics  ${\rmbc}(\{\axci,\axcf\})$, ${\rmbc}(\{\axcl,\axcf\})$ and ${\rmbc}(\{\axci, \axcl,\axcf\})$ coincide.
\end{prop} 
\begin{proof} Recall from Remark \ref{obsRmbC-ext}~(3) that $\vdash_{\rmbc}(\alpha \wedge \neg\alpha) \to \neg\cons\alpha$ and $\vdash_{\rmbc}\cons\alpha \to \neg(\alpha \wedge \neg\alpha)$. Then, for every BALFI \B\ with carrier $A$, and every $a \in A$, $(a \wedge \neg a) \leq \neg\cons a$ and $\cons a \leq \neg(a \wedge \neg a)$.    Let $\B \in \mathbb{BI}(\{\axci,\axcf\})$, and let $a \in A$. Then $a \land \neg a=\neg\cons a$ and so $\neg(a \land \neg a)=\neg\neg\cons a \leq \cons a$. Therefore $\B \in \mathbb{BI}(\{\axcl,\axcf\})$.  Conversely, suppose that $\B \in \mathbb{BI}(\{\axcl,\axcf\})$ and let $a \in A$. Since $\cons a = \neg(a \wedge \neg a)$ then $\neg\cons a = \neg\neg(a \wedge \neg a) \leq (a \wedge \neg a)$. From this, $\B \in \mathbb{BI}(\{\axci,\axcf\})$. This shows the first part of the Proposition. The second part follows from Theorem~\ref{adeq-ext-BALFI}.
\end{proof}

\begin{exem} [BALFIs for \rmbcciw] \label{def-rmbcciw}
The logic \mbc(\axwci) was considered in~\cite{CC16} under the name \mbcciw. This logic was introduced in~\cite{avr:05} under the name ${\bf B}[\{({\bf k1}),({\bf k2})\}]$, presented by means of a standard Gentzen calculus such that {\bf B} is a Gentzen calculus  for \mbc. The logic \mbcciw\ is the least extension of \mbc\ in which the consistency connective can be defined in terms of the other connectives, namely: $\cons \alpha$ is equivalent to $\sneg(\alpha \wedge \neg\alpha)$.\footnote{Recall from Remark~\ref{obsRmbC-ext}~(2) that $\sneg$ denotes the classical negation definable in \mbc\ as $\sneg\alpha = \alpha \to \bot$, where $\bot$ denotes any formula of the form $\beta \wedge \neg \beta \wedge \cons\beta$. Thus, rigorously speaking, $\circ$ {\em is not} defined in terms of the other connectives, since $\circ$ is essential on order to define $\bot$. So, the right signature for \mbcciw\ and its extensions is  $\Sigma_{C_0}$, introduced in Definition~\ref{signa}.}
Let \rmbcciw\ be the logic \rmbc(\axwci). Because of the satisfaction of the replacement  property, and given that the consistency connective can be defined in terms of the other connectives, the connective $\cons$ can be eliminated from the signature, and so we consider the logic \rmbcciw\  defined over the signature $\Sigma_{C_0}$ (recall Definition~\ref{signa}), obtained from \cplp\ by adding  $(\axtnd)$, $(\Rn)$, and  axiom schema $({\bf Bot})$: $\bar 0 \to \alpha$. In such presentation of \rmbcciw, the strong negation is defined by the formula $\sneg\alpha=\alpha \to \bar 0$. The algebraic models for this presentation of \rmbcciw\  are given by BALFIs $\B = \langle A, \wedge, \vee, \to,\neg,0,1 \rangle$ over $\Sigma_{C_e}$ such that its reduct $\mathcal{A}=\langle A, \wedge, \vee, \to,0,1 \rangle$ to $\Sigma_{BA}$ is a Boolean algebra and the unary operator $\neg$ satisfies $a \vee \neg a = 1$ for every $a \in A$. On the other hand, the term $\cons a$ is an abbreviation for $\sneg(a \land \neg a)$ in such BALFIs. 
\end{exem}

It is also interesting to observe that $\cons$ satisfies a sort of \textit{necessitation rule} (by analogy to modal logics) in certain extensions of \rmbc:

\begin{prop} Consider the {\em necessitation rule} for \cons: 
\[\frac{\alpha}{\cons\alpha}  \hspace{8mm}  (Nec_\cons)\]
Then, $(Nec_\cons)$ is an admissible rule in $\rmbc(\{\axcl,\axce\})$.\footnote{Recall that a structural inference rule is  {\em admissible} in a logic {\bf L} if the following holds: whenever the premises of an instance of the rule are theorems of {\bf L}, then the conclusion of the same instance of the rule is a theorem of {\bf L}.}
\end{prop} 
\begin{proof} Assume that $\vdash_{\rmbc(\{\axcl,\axce\})} \alpha$.  By the rules of \cplp\ it follows that  $\vdash_{\rmbc(\{\axcl,\axce\})} \beta \sse (\alpha \wedge \beta)$ for every formula $\beta$. In particular, $\vdash_{\rmbc(\{\axcl,\axce\})} \neg\alpha \sse (\alpha \wedge \neg\alpha)$ and so, by (\Rn), $\vdash_{\rmbc(\{\axcl,\axce\})} \neg\neg\alpha \sse \neg(\alpha \wedge \neg\alpha)$. On the other hand, from    $\vdash_{\rmbc(\{\axcl,\axce\})} \alpha$ it follows that   $\vdash_{\rmbc(\{\axcl,\axce\})} \neg\neg\alpha$, by (\axce) and (\MP). Then  $\vdash_{\rmbc(\{\axcl,\axce\})} \neg(\alpha \wedge \neg\alpha)$. By (\axcl) and (\MP) we conclude that $\vdash_{\rmbc(\{\axcl,\axce\})} \cons\alpha$.
\end{proof}

\

We  can now provide a solution to the first open problem posed in~\cite{CM}:

\begin{exem} [A  paraconsistent extension of \bc\ with replacement]  \label{exem-bC} \ \\
Consider the logic \bc\ introduced in~\cite{CM}. By using the notation proposed above, \bc\ corresponds to \mbc(\axcf).\footnote{We will write $\mbc(\varphi)$, $\rmbc(\varphi)$ and $\mathbb{BI}(\varphi)$ instead of  $\mbc(\{\varphi\})$, $\rmbc(\{\varphi\})$ and $\mathbb{BI}(\{\varphi\})$, respectively.} Then \rbc\ (that is,  \rmbc(\axcf)) is an extension of \bc\ which satisfies replacement (by Theorem~\ref{IPE-rmbc}) while being still paraconsistent. Moreover, \rbc\ is an \lfi. These facts  can be easily proven by using the BALFI $\B'$ defined over $\wp(\{w_1,w_2\})$ considered in Remark~\ref{obsRmbC} (see Figure~\ref{f00}).
\begin{center}
	\begin{figure}[h!]
		\centering
		\includegraphics[width=0.2\textwidth]{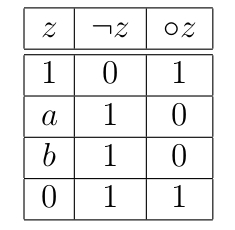}
		\caption{A 4-elements paraconsistent BALFI for \rbc}\label{f00}
	\end{figure}
\end{center}
\noindent In fact, it is immediate to see that $\B'$ is a model of  $(\axcf)$. It is worth noting that  $\B'$ is {\em not} a model of  $(\axwci)$: $0=\cons a \neq \sneg(a\wedge \neg a)=\sneg a =b$. Therefore,  $\B'$ is neither a model of  $(\axci)$ nor of $(\axcl)$, given that any of these axioms implies $(\axwci)$.
\end{exem}

We can now offer a solution to the second open problem posed in~\cite{CM}:

\begin{exem} [A paraconsistent extension of \ci\ with replacement] \label{exem-Ci} 
Now, consider the logic \ci\ introduced in~\cite{CM}, which corresponds to \mbc(\{\axcf, \axci\}), and let   \rci=\rmbc(\{\axcf, \axci\}). By Proposition~\ref{collapse}, \rci\  also derives the schema $(\axcl)$. It can be proven that \rci\ is an extension of \ci\ which satisfies replacement (by Theorem~\ref{IPE-rmbc}) while it is still paraconsistent. In order to prove this, consider the BALFI $\B''$ defined over the Boolean algebra $\A_{16}=\wp(W_4)$, the powerset of $W_4=\{w_1,w_2,w_3,w_4\}$ displayed in Figure~\ref{f01}. In that figure, $X$ is different from $\{w_1,w_2\}$ and $\{w_3,w_4\}$ (note that $0=\emptyset$ and $1=W_4$). It is immediate to see that $\B''$ is a BALFI for \rci. Hence, using this model it follows easily that \rci\ is a paraconsistent extension of \ci\ which satisfies (IpE) and $(\axcl)$.

Another paraconsistent model for \rci\  defined over $\A_{16}$ is displayed in Figure~\ref{f02}. Here, the cardinal of $X$ is different from $2$.\footnote{It is worth noting that  many of the experiments leading to the	generation of the two models presented here were carried out  with the help of the model finder {\em Nitpick} \citep{Nitpick}, which is part of the automated tools integrated into the proof assistant  Isabelle/HOL~\citep{nipkow2002isabelle}; recall also the comments in Footnote~\ref{fn:isabelle}.}

\begin{center}
\begin{figure}[h!]
  \centering
    \includegraphics[width=0.55\textwidth]{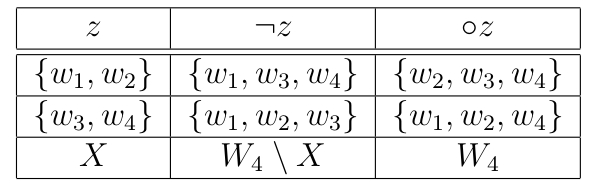}
    \caption{A 16-elements paraconsistent BALFI for \rci}\label{f01}
\end{figure}
\end{center}

\begin{center}
\begin{figure}[h!]
  \centering
    \includegraphics[width=0.55\textwidth]{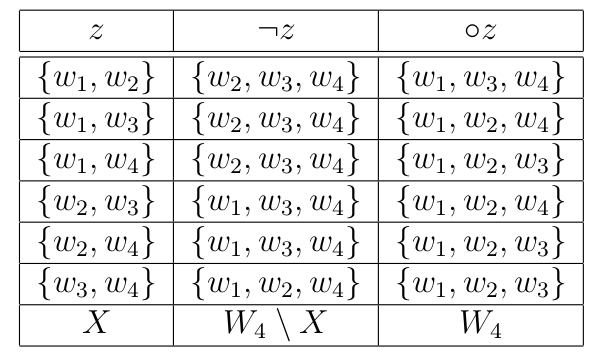}
    \caption{Another 16-elements paraconsistent BALFI for \rci}\label{f02}
\end{figure}
\end{center}

\end{exem}

\begin{exem}
We can offer now a model of \rmbc(\axcl) that does not satisfy axiom $(\axcf)$.  For this purpose  consider the BALFI $\B'''$ defined over the Boolean algebra $\A_4=\wp(\{w_1,w_2\}) = \{0, a,b,1\}$ according to Example~\ref{allBALFIs}, and  displayed in Figure~\ref{f03}. 
Observe that $\B''' \Vdash \axcl$. However, $\B'''$ is not a model of $(\axcf)$: $\neg\neg 0 = a \not\leq 0$.

\begin{center}
\begin{figure}[h!]
\vspace*{-0.25cm}
  \centering
    \includegraphics[width=0.2\textwidth]{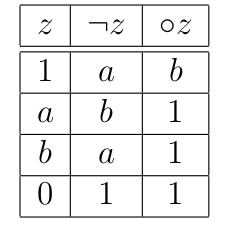}
    \caption{A paraconsistent BALFI model of \rmbc(\axcl) that does not satisfy axiom $(\axcf)$}\label{f03}
\end{figure}
\end{center}

 
\end{exem}

\section{Limits for replacement while preserving paraconsistency} \label{limits}

In~\citet[Theorem~3.51]{CM} some sufficient conditions were given to show that certain extensions of \bc\ and \ci\ cannot satisfy replacement while being still paraconsistent. This result shows that there are limits, much before reaching classical logic \cpl,  for extending \rmbc\ while preserving paraconsistency. This result will  now  be applied   in order to obtain two important examples of \lfis\ that   cannot be extended with replacement at risk of  losing their character of paraconsistency.

The first example, to begin with,  is in fact  a family of 8,192 examples: 

\begin{exem} [Three-valued \lfis] \ \\  \label{3-valued} 
Recall the family of 8Kb three-valued \lfis\ introduced by Marcos in an unpublished draft,  and discussed in~\citet[Section~3.11]{CM} and in~\citet[Section~5.3]{CCM}. As it was observed in these references, the schema ${\neg(\alpha \wedge \neg\alpha)}$ is valid in all of these logics. In addition, all these logics are models of axioms $(\axci)$ and $(\axcf)$, i.e., they extend \ci~\citep[see, e.g.,][Theorem~130]{CCM}. 
But in~\citet[Theorem~3.51(ii)]{CM} it was proved that $(IpE)$ cannot hold in any paraconsistent extension of \ci\ in which the schema $\neg(\alpha \wedge \neg\alpha)$ is valid. As a consequence, the inference rules $(\Rn)$ and $(\Rb)$ cannot be added to any of them at the  risk of losing paraconsistency. Indeed, if {\bf L} is any of such three-valued logics, the corresponding logic  {\bf RL} obtained by adding both rules will derive the axiom schema $\cons\alpha$ (the proof of this fact can be easily adapted from the one for Theorem~3.51(ii) presented in~\cite{CM}). But then, the negation $\neg$ is explosive in {\bf RL}, by $(\axexp)$ and $(\MP)$. This shows that these three-valued \lfis, including the well-known da Costa and D'Ottaviano's logic \dacdot\ (and so all of its equivalent presentations, such as \lfium, {\bf CLuN} or \lptz), as well as Sette's logic \set, if extended by the inference rules proposed here,  will be no longer paraconsistent. Of course, this result is related to the one obtained in~\cite{avr:bez:17}, which states  that for no three-valued paraconsistent logic with implication the replacement property can hold.
\end{exem}

The second example deals with the well-known da Costa logic $C_1$.

\begin{exem} [da Costa's logic $C_1$] \ \\ \label{C1}
Newton da Costa introduced in~\cite{dC63} his famous hierarchy of paraconsistent systems $C_n$ (for $n \geq 1$), the first systematic study on paraconsistency in the literature. As discussed above, da Costa's approach was generalized through the notion of \lfis. The first and stronger system in da Costa's hierarchy is $C_1$, which is equivalent (up to language) with \cila. The logic \cila\ corresponds, with the notation introduced above, to $\mbc(\{\axci,\axcl,\axcf,\axca_\wedge,\axca_\vee,\axca_\to\})$. Now, let  \rcila\ be the logic  $\rmbc(\{\axci,\axcl,\axcf,\axca_\wedge,\axca_\vee,\axca_\to\})$; then, this logic derives $(\axwci)$. Indeed, as shown in~\citet[Proposition~3.1.10]{CC16}, axiom $(\axwci)$ is derivable from \mbc\ plus axiom $(\axci)$. This being so, by Example~\ref{def-rmbcciw} and the fact that $\bot \defin (\alpha \wedge \neg \alpha) \wedge \neg(\alpha \wedge \neg \alpha)$ is a bottom formula in \cila\ (hence in \rcila) for any $\alpha$ (i.e., $\bot$ implies any other formula), the connective $\cons$ could be eliminated from the signature of \rcila, and so the logic \rcila\ could be  defined over the signature $\Sigma_C$ (recall Definition~\ref{signa}). In that case, \rcila\ would coincide with ${\bf R}C_1$, the extension of $C_1$ by adding  the inference rule $(\Rn)$ (and where the notion of derivation is given as in Definition~\ref{deriv-RmbC}). The question is to find a model of \rcila\ (or, equivalently, of ${\bf R}C_1$) which is still paraconsistent.

In~\citet[Theorem~3.51(iv)]{CM} it was proved that $(IpE)$ cannot hold in any paraconsistent extension of \ci\ in which the schema $({\bf dm})$: $\neg(\alpha \wedge \beta) \to (\neg\alpha \vee \neg\beta)$ is valid. On the other hand, in~\citet[Theorem~3.6.4]{CC16} it was proved that the logic obtained from \mbcciw\ by adding $(\axca_\wedge)$ is equivalent to the logic obtained from \mbcciw\ by adding the schema axiom $({\bf dm})$. Since \rcila\ derives $(\axwci)$ and  $(\axca_\wedge)$, it also derives the schema  $({\bf dm})$. Finally, given that \rcila\ is an extension of \ci\ which satisfies $(IpE)$ and where $({\bf dm})$ is valid, it cannot be paraconsistent, by~\citet[Theorem~3.51(iv)]{CM}. 
\end{exem}

The failure of finding a paraconsistent extension of $C_1$ enjoying (IpE) was first observed in~\citet[Chapter~5, Theorem~9]{urbas}, by proving that the extension of $C_n$ by replacement collapses to classical logic (see also~\citet[Theorem~9]{urbas:89} and~\citet[Section~4]{sylvan}).

\section{Neighborhood semantics for  \rmbc} \label{sec:nbhd}

In this section we will introduce a particular case of BALFIs based on powerset Boolean algebras. These structures are  in fact equivalent to  neighborhood frames for non-normal modal logics, as we shall see in Section~\ref{non-normal}. 
Moreover, we shall prove in that Section that, with this semantics, \rmbc\ can be defined within the bimodal version of the {\em minimal modal logic}  {\bf E} (also termed {\em classical modal logic} in~\citet[Definition~8.1]{che:80}).


\begin{defn} \label{neighframe} Let $W$ be a non-empty set. A {\em neighborhood frame} for \rmbc\ over $W$ is a triple $\matF=\langle W, S_\neg,S_\cons\rangle$ such that $S_\neg:\wp(W) \to \wp(W)$ and  $S_\cons:\wp(W) \to \wp(W)$ are functions. A {\em neighborhood model} for \rmbc\ over \matF\ is a pair $\matM=\langle \matF, d\rangle$ such that \matF\ is a neighborhood frame for \rmbc\ over $W$ and $d:\mathcal{V} \to \wp(W)$ is a (valuation) function (given a denumerable set $\mathcal{V}$ of propositional variables).
\end{defn}

\begin{defn} \label{defneighrmbc} Let $\matM=\langle \matF,d\rangle$ be a neighborhood model for \rmbc\ over $\matF=\langle W, S_\neg,S_\cons\rangle$. It induces a {\em denotation function} $\termvalue{\cdot}_\matM:For(\Sigma) \to \wp(W)$ defined recursively as follows (for simplicity, we will write $\termvalue{\varphi}$ instead of $\termvalue{\varphi}_\matM$ when \matM\ is clear from the context):\\[1mm]
\indent $\begin{array}{lll}
\termvalue{p} &=& d(p), \mbox{ if $p \in \mathcal{V}$};\\[2mm]
\termvalue{\varphi \wedge \psi} &=& \termvalue{\varphi}\cap \termvalue{\psi};\\[2mm]
\termvalue{\varphi \vee \psi} &=& \termvalue{\varphi}\cup \termvalue{\psi};\\[2mm]
\termvalue{\varphi \to \psi} &=& \termvalue{\varphi}\to \termvalue{\psi} = (W \setminus \termvalue{\varphi}) \cup \termvalue{\psi};\\[2mm]
\termvalue{\neg\varphi} &=& (W \setminus \termvalue{\varphi}) \cup S_\neg(\termvalue{\varphi}); \mbox{ and}\\[2mm]
\termvalue{\cons\varphi} &=& (W \setminus (\termvalue{\varphi} \cap \termvalue{\neg\varphi})) \cap S_\cons(\termvalue{\varphi}) = (W \setminus (\termvalue{\varphi} \cap S_\neg(\termvalue{\varphi}))) \cap S_\cons(\termvalue{\varphi}).
\end{array}$
\end{defn}

Clearly $\termvalue{\varphi} \cup \termvalue{\neg\varphi} = W$, but $\termvalue{\varphi} \cap \termvalue{\neg\varphi}$ is not necessarily empty.  In addition, $\termvalue{\varphi} \cap \termvalue{\neg\varphi} \cap \termvalue{\cons\varphi} = \emptyset$.
Intuitively, the functions $S_\neg$ and $S_\cons$ act as `regulators' for the values given to the non-standard connectives $\neg$ and $\cons$. Indeed, $S_\neg(\termvalue{\varphi})$ is the value to be `added' to the value of the classical negation of $\varphi$ (that is, the Boolean complement of $\termvalue{\varphi}$) in order to allow paraconsistency.
In fact, in \rmbcciw\ and its extensions the value of $\cons\varphi$ is the Boolean complement of the value of the contradiction $\varphi \land \neg \varphi$. However, in \rmbc\ the value of $\cons\varphi$ could be less than this, and thus it becomes `decreased' by intersection with $S_\cons(\termvalue{\varphi})$.

\begin{defn} \label{consneigh} Let $\matM=\langle \matF,d\rangle$ be a neighborhood model for \rmbc.\\
(1) We say that a  formula $\varphi$ is {\em valid} (or {\em true}) {\em in}  \matM, denoted by $\matM \Vdash \varphi$, if $\termvalue{\varphi} = W$.\\
(2) We say that a formula $\varphi$ is {\em valid w.r.t. neighborhood models}, denoted by $\models_\mathbb{NM} \varphi$, if  $\matM \Vdash \varphi$ for every neighborhood model \matM\ for \rmbc.\\
(3) The consequence relation $\models_\mathbb{NM}$ induced by neighborhood models for  \rmbc\ is defined as follows: $\Gamma \models_\mathbb{NM} \varphi$ if either  $\models_\mathbb{NM} \varphi$
or there exists a finite, non-empty subset $\{\gamma_1,\ldots,\gamma_n\}$ of $\Gamma$ such that ${ \models_\mathbb{NM}(\gamma_1 \wedge \ldots \wedge \gamma_n) \to \varphi}$.
\end{defn}

Clearly, $\Gamma \models_\mathbb{NM} \varphi$ if either $\termvalue{\varphi}_\matM = W$ for every neighborhood model $\matM$ for \rmbc, 
or there exists a finite, non-empty subset $\{\gamma_1,\ldots,\gamma_n\}$ of $\Gamma$ such that, for every neighborhood model \matM\ for \rmbc, $\bigcap_{i=1}^n\termvalue{\gamma_i}_\matM \subseteq  \termvalue{\varphi}_\matM$.

\begin{prop} \label{kripe-BALFI}
Given a neighborhood frame $\matF=\langle W, S_\neg,S_\cons\rangle$ for \rmbc\ let $\tilde{\neg},\tilde{\cons}:\wp(W) \to \wp(W)$ defined as follows: $\tilde{\neg}(X)=(W \setminus X) \cup S_\neg(X)$ and  $\tilde{\cons}(X)=(W \setminus (X \cap S_\neg(X)) \cap S_\cons(X)$.  Then $\B_\matF\defin\langle \wp(W),\cap,\cup,\to,\tilde{\neg},\tilde{\cons},\emptyset,W\rangle$ is a BALFI. Moreover, if $\matM=\langle \matF,d\rangle$ is a neighborhood model for \rmbc\ over $\matF=\langle W, S_\neg,S_\cons\rangle$ then the denotation function $\termvalue{\cdot}_\matM$ is a valuation over $\B_\matF$.
\end{prop}
\begin{proof}
It is immediate from the definitions.
\end{proof}

\begin{coro}  [Soundness of \rmbc\ w.r.t.~neighborhood models] \label{soundneighrmbc} \ \\
If $\Gamma \vdash_{\rmbc} \varphi$ then $\Gamma \models_\mathbb{NM} \varphi$.
\end{coro}
\begin{proof}
It follows from soundness of \rmbc\ w.r.t.~BALFI semantics (Theorem~\ref{sound-BALFI}) and by Proposition~\ref{kripe-BALFI}.
\end{proof}

\

Proposition~\ref{kripe-BALFI} suggests the following :

\begin{defn} \label{valid-frame}
Let $\matF=\langle W, S_\neg,S_\cons\rangle$ be a neighborhood frame for \rmbc. A formula $\varphi$ is {\em valid in \matF} if  $\matM \Vdash \varphi$ for every neighborhood model $\matM=\langle \matF,d\rangle$  for \rmbc\ over \matF.
\end{defn}

In order to prove completeness of \rmbc\ w.r.t.~neighborhood models, Stone's representation theorem for Boolean algebras  will be used. This important theorem, first proved in~\cite{stone}, states that every Boolean algebra is isomorphic to a Boolean subalgebra of $\wp(W)$, for a suitable $W$. This means that, given a Boolean algebra $\A$, there exists a set $W$ and an injective homomorphism $i:\A \to \wp(W)$ of Boolean algebras. Note that $i(a)=W$ if and only if  $a= 1$.

\begin{teom} [Completeness of \rmbc\ w.r.t.~neighborhood models] \ \\ \label{BALFI-kripke} 
If $\Gamma \models_\mathbb{NM} \varphi$ then $\Gamma \vdash_{\rmbc} \varphi$.
\end{teom}
\begin{proof}
Let  $\mathcal{A}_{can}$ be the Boolean algebra with carrier  $A_{can}=For(\Sigma)/_{\equiv}$, and let  $\mathcal{B}_{can}$ be the corresponding expansion to $\Sigma_e$, as defined in the proof of Theorem~\ref{comple-BALFI}.  Let $i:\mathcal{A}_{can} \to \wp(W)$ be an injective homomorphism of Boolean algebras, according to Stone's theorem as discussed above. Consider the neighborhood frame  $\matF_{can}=\langle W, S_\neg,S_\cons\rangle$  for \rmbc\ such that the functions $S_\neg$ and $S_\cons$ satisfy the following:
$S_\neg(i([\alpha])) = i([\neg\alpha])$, and $S_\cons(i([\alpha])) = {i}([\cons\alpha])$, for every formula $\alpha$ (where $[\alpha]$ denotes the equivalence class of $\alpha$ w.r.t.~${\equiv}$). Observe that these functions are well-defined, since every connective in \rmbc\ is congruential and $i$ is injective. The values of these functions outside the image of $i$ can be arbitrary.
Now, let $\matM_{can}=\langle \matF_{can},d_{can}\rangle$ be the neighborhood model for \rmbc\  such that $d_{can}(p)\defin{i}([p])$, for every propositional variable $p$.\\[1mm]
{\bf Fact:} $\termvalue{\alpha}_{\matM_{can}}=i([\alpha])$, for every formula $\alpha$.\\[1mm] 
The proof of the Fact will be done by induction on the complexity of the formula $\alpha$ (recalling Definition~\ref{defneighrmbc} for $\termvalue{\cdot}_\matM$). The complexity of $\cons\alpha$ is taken to be stricty greater than the complexity of $\neg\alpha$. The case for $\alpha$ atomic or $\alpha=\beta \# \gamma$ for $\# \in \{\wedge,\vee,\to\}$ is clear, by the very definitions and by induction hypothesis. Now, suppose that $\alpha=\neg\beta$. By induction hypothesis, $\termvalue{\beta}=i([\beta])$. Observe that $\sneg[\beta] \leq [\neg\beta]$ in $\mathcal{A}_{can}$ (where $\sneg$ denotes the Boolean complement in $\mathcal{A}_{can}$), hence $W\setminus i([\beta]) = i(\sneg[\beta]) \subseteq i([\neg\beta])$. Thus, \\

$\termvalue{\neg\beta}= (W \setminus \termvalue{\beta}) \cup S_\neg(\termvalue{\beta})=(W \setminus i([\beta])) \cup S_\neg(i([\beta]))$\\

\hspace{0.8cm} $=(W \setminus i([\beta])) \cup i([\neg\beta]) = i([\neg\beta])$.\\

\noindent Finally, let $\alpha=\cons\beta$. Since $[\cons\beta] \leq \sneg([\beta] \wedge [\neg\beta])$ in $\mathcal{A}_{can}$ then $i([\cons\beta]) \subseteq W \setminus(i([\beta]) \cap i([\neg\beta]))$. Hence, by induction hypothesis,\\

$\termvalue{\cons\beta}= (W \setminus (i([\beta]) \cap i([\neg\beta]))) \cap S_\cons(i([\beta])$\\

\hspace{0.8cm} $=(W \setminus (i([\beta]) \cap i([\neg\beta]))) \cap i([\cons\beta]) = i([\cons\beta])$.\\

\noindent This concludes the proof of the Fact.

Because of the Fact, $\matM_{can} \Vdash \alpha$ iff $i([\alpha])=W$ iff $[\alpha]=1$ iff $\vdash_{\rmbc} \alpha$. Now, suppose that $\Gamma \models_\mathbb{NM} \varphi$. 
If $\models_\mathbb{NM} \varphi$ then, in particular, $\matM_{can} \Vdash \varphi$ and so $\vdash_{\rmbc} \varphi$. From this, $\Gamma \vdash_{\rmbc} \varphi$. On the other hand, suppose that there exists a finite, non-empty subset $\{\gamma_1,\ldots,\gamma_n\}$ of $\Gamma$ such that $\models_\mathbb{NM} (\gamma_1 \wedge \ldots \wedge \gamma_n) \to \varphi$. By reasoning as above, it follows that $\vdash_{\rmbc} (\gamma_1 \wedge \ldots \wedge \gamma_n) \to \varphi$ and so $\Gamma \vdash_{\rmbc} \varphi$ as well.
\end{proof}   

\subsection{\rmbc\ is definable within the minimal bimodal modal logic} \label{non-normal}

This  section  shows  that \rmbc\ is definable within the bimodal version of the  minimal modal logic  {\bf E}, also called  classical modal logic in~\citet[Definition~8.1]{che:80}. In terms of combination of modal logics, this bimodal logic is equivalent to the {\em fusion} (or, equivalently, the {\em constrained fibring} by sharing the classical connectives)  of {\bf E} with itself.\footnote{For the basic notions of combining logics the reader can consult~\cite{CC07} and \cite{car:etal:2008}.} This means that the minimal non-normal modal logic with two independent modalities $\square_1$ and  $\square_2$, which will be denoted by ${\bf E}{\oplus}{\bf E}$, contains \rmbc, the minimal self-extensional \lfi. As  we shall see, in the case of {\rmbc} both modalities are required for defining the two non-classical conectives $\neg$ and $\circ$. For {\rmbcciw} and its extensions only one modality is needed, since $\circ$ can be taken as a derived operator (recall Example~\ref{def-rmbcciw}). Hence {\rmbcciw} is definable in \textbf{E}.

The minimal (or classical) modal logic {\bf E} was introduced by~\cite{montague} and~\cite{scott} with semantics given by neighborhood models, which generalize the usual Kripke models for modal logic.
Neighborhood semantics has become one of the main mechanisms for providing semantics for non-normal modal logics as well as for some normal modal logics that are not complete w.r.t.~Kripke models \citep[see e.g.][]{gabbay1975normal}. Among the first important applications of neighborhood semantics we count the well-known book on counterfactual conditionals by \citet{d.lewis}. Ever since, it has found many successful applications in areas of conditional, deontic, and epistemic logic \citep[see, e.g.,][]{pac:2007}.
For convenience, the definition of modal logic {\bf E} will be briefly  surveyed below.

\begin{defn} [\cite{che:80}, Definition~7.1]
A {\em minimal model} is a triple $\mathcal{N} = \langle W, N, d \rangle$ such that $W$ is a non-empty set and $N:W \to \wp(\wp(W))$ and $d:\mathcal{V} \to \wp(W)$ are functions. The class of minimal models will be denoted by $\mathcal{C_M}$.
\end{defn}

Let $\Sigma_m=\{\land, \lor, \to, \sneg, \square, \lozenge\}$ and $\Sigma_{bm}=\{\land, \lor, \to, \sneg, \square_1,\lozenge_1,\square_2,\lozenge_2\}$ be the signatures introduced in  Definition~\ref{signa}. The class of models  $\mathcal{C_M}$ induces a modal consequence relation defined as follows:

\begin{defn} [\cite{che:80}, Definition~7.2] \label{denotmod}
Let $\mathcal{N}$ be a minimal model and $w \in W$. $\mathcal{N}$ is said to {\em satisfy a formula  $\varphi \in For(\Sigma_m)$   in $w$},
denoted by $\models^\mathcal{N}_w \varphi$, according to the following recursive definition (here $\termvalue{\varphi}^{\mathcal{N}}$ stands for the set $\{w \in W \ : \  \models^\mathcal{N}_w \varphi\}$, the {\em denotation} of $\varphi$ in $\mathcal{N}$):\\

$\begin{array}{ll}
1. & \mbox{ if $p$ is a propositional variable then $\models^\mathcal{N}_w p$ \ iff \ $w \in d(p)$};\\[2mm]
2. & \models^\mathcal{N}_w \sneg \alpha \ \mbox{ iff } \ \not\models^\mathcal{N}_w \alpha;\\[2mm]
3. & \models^\mathcal{N}_w \alpha \land \beta \ \mbox{ iff } \ \models^\mathcal{N}_w \alpha \ \mbox{ and } \ \models^\mathcal{N}_w \beta;\\[2mm]
4. & \models^\mathcal{N}_w \alpha \lor \beta \ \mbox{ iff } \ \models^\mathcal{N}_w \alpha \ \mbox{ or } \ \models^\mathcal{N}_w \beta;\\[2mm]
5. & \models^\mathcal{N}_w \alpha \to \beta \ \mbox{ iff } \ \not\models^\mathcal{N}_w \alpha \ \mbox{ or } \ \models^\mathcal{N}_w \beta;\\[2mm]
6. & \models^\mathcal{N}_w \square \alpha \ \mbox{ iff } \ \termvalue{\alpha}^{\mathcal{N}} \in N(w);\\[2mm]
7. & \models^\mathcal{N}_w \lozenge \alpha \ \mbox{ iff } \ (W \setminus\termvalue{\alpha}^{\mathcal{N}}) \notin N(w).
\end{array}$
\end{defn}

A formula $\varphi$ is true in  $\mathcal{N}$ if $\termvalue{\varphi}^{\mathcal{N}} = W$, and it is valid w.r.t. $\mathcal{C_M}$, denoted by $\models^\mathcal{C_M} \varphi$, if it is true in every minimal model. The degree-preserving consequence w.r.t. $\mathcal{C_M}$ can be defined analogously  to the one for neighborhood semantics for \rmbc\ given in Definition~\ref{consneigh}. Namely,
$\Gamma \models^\mathcal{C_M} \varphi$ if either $\models^\mathcal{C_M} \varphi$, 
or there exists a finite, non-empty subset $\{\gamma_1,\ldots,\gamma_n\}$ of $\Gamma$ such that $\models^\mathcal{C_M}(\gamma_1 \wedge \ldots \wedge \gamma_n) \to \varphi$. The latter is equivalent to  $\bigcap_{i=1}^n \termvalue{\gamma_i}^{\mathcal{N}} \subseteq \termvalue{\varphi}^{\mathcal{N}}$.

\begin{defn}  [\cite{che:80}, Definition~8.1] \label{logiE}
The {\em minimal modal logic} (or {\em  classical modal logic})  {\bf E} is defined by means of a Hilbert calculus over the signature $\Sigma_m$ obtained by adding to the Hilbert calculus for \cplp\ (recall Definition~\ref{hilCPLP}) the following axiom schemas and rules:

  \begin{gather}
    \alpha \lor \sneg \alpha                        \tag{\axpem}\\[2mm]
    \alpha \imp \big(\sneg \alpha \imp \beta\big)
                                                     \tag{\axexpl}\\[2mm]
    \lozenge\alpha \leftrightarrow \sneg\square\sneg\alpha
                                                     \tag{\axmod}\\[2mm]
    \frac{\alpha \sse \beta}{\square\alpha \sse \square\beta}  \tag{\RE}                                                    
  \end{gather}
\end{defn}

\

The notion of derivation in {\bf E} is defined as for \rmbc, recall Definition~\ref{deriv-RmbC}.
Note that (\axpem) and (\axexpl), together with \cplp, guarantee that {\bf E} is an expansion of propositional classical logic by adding the modalities $\square$ and $\lozenge$ which are interdefinable as usual, and such that both are congruential. That is, {\bf E} satisfies replacement.

\begin{teom}  [\cite{che:80}, Section~9.2]
The logic {\bf E} is sound and complete w.r.t. the semantics in  $\mathcal{C_M}$, namely: $\Gamma \vdash_{\bf E} \varphi$ \ iff \ $\Gamma \models^\mathcal{C_M} \varphi$.
\end{teom}

\begin{defn} [Minimal bimodal logic] The {\em minimal bimodal logic} ${\bf E} {\oplus} {\bf E}$  is defined by means of a Hilbert calculus over signature $\Sigma_{bm}$ obtained by adding to the Hilbert calculus for \cplp\ the following axiom schemas and rules, for $i=1,2$:

  \begin{gather}
    \alpha \lor \sneg \alpha                        \tag{\axpem}\\[2mm]
    \alpha \imp \big(\sneg \alpha \imp \beta\big)
                                                     \tag{\axexpl}\\[2mm]
    \lozenge_i\alpha \leftrightarrow \sneg\square_i\sneg\alpha
                                                     \tag{\axmod$_i$}\\[2mm]
    \frac{\alpha \sse \beta}{\square_i\alpha \sse \square_i\beta}  \tag{\RE$_i$}                                                    
  \end{gather}
\end{defn}

\

Observe that ${\bf E} {\oplus} {\bf E}$ is obtained from {\bf E} by `duplicating' the modalities. There is no relationship between $\square_1$ and $\square_2$ and so $\lozenge_1$ and $\lozenge_2$ are also independent. 

The semantics of ${\bf E} {\oplus} {\bf E}$ is given by the class $\mathcal{C'_M}$ of structures of the form $\mathcal{N} = \langle W, N_1,N_2,d \rangle$ such that $W$ is a non-empty set and $N_i:W \to \wp(\wp(W))$ (for $i=1,2$) and $d:\mathcal{V} \to \wp(W)$ are functions. The denotation $\termvalue{\varphi}^{\mathcal{N}}$ of a formula $\varphi \in For(\Sigma_{bm})$ in $\mathcal{N}$ is defined by an obvious adaptation of Definition~\ref{denotmod} to $For(\Sigma_{bm})$. By  defining the consequence relations $\vdash_{{\bf E} {\oplus} {\bf E}}$ and 
$\models^\mathcal{C'_M}$ in analogy to the ones for {\bf E}, it is straightforward to adapt the proof of soundness and completeness of {\bf E} to the bimodal case. 

\begin{teom} \label{compleE+E}
The logic ${\bf E} {\oplus} {\bf E}$ is sound and complete w.r.t. the semantics in  $\mathcal{C'_M}$, namely: $\Gamma \vdash_{{\bf E} {\oplus} {\bf E}} \varphi$ \ iff \ $\Gamma \models^\mathcal{C'_M} \varphi$.
\end{teom}

From the point of view of combining logics,  ${\bf E} {\oplus} {\bf E}$ is the fusion (or, equivalently, the constrained fibring by sharing the classical connectives)  of {\bf E} with itself \citep[see][]{CC07,car:etal:2008}.

Finally, it will be shown that \rmbc\ can be defined inside ${\bf E}{\oplus} {\bf E}$ by means of the following abbreviations, which are motivated by Definition~\ref{defneighrmbc} (having in mind that $\square_1$ and $\square_2$ will be interpreted by the mappings $S_\neg$ and $S_\circ$, respectively):
$$\neg \varphi \defin \varphi \to \square_1 \varphi \ \mbox{ and } \ \cons \varphi \defin \sneg(\varphi \land \square_1 \varphi) \land   \square_2 \varphi.$$
In order to see this, observe that any function $N:W \to \wp(\wp(W))$ induces a unique function $S:\wp(W) \to \wp(W)$ given by $S(X) = \{w \in W \ : \ X \in N(w)\}$. Conversely, any function $S:\wp(W) \to \wp(W)$ induces a function $N:W \to \wp(\wp(W))$ given by $N(w) = \{X \subseteq W \ : \ w \in S(X)\}$. Both assignments $N \mapsto S$ and $S \mapsto N$ are  inverses of each other. From this, a structure (or minimal model) $\mathcal{N} = \langle W, N_1,N_2, d \rangle$ for  ${\bf E} {\oplus} {\bf E}$ can be transformed into a neighborhood model $\matM=\langle W, S_\neg,S_\cons,d\rangle$ for \rmbc\ such that $S_\neg$ and $S_\cons$ are  obtained, respectively, from the functions  $N_1$ and $N_2$ as indicated above. Observe that
$$w \in  \termvalue{\square_1\varphi}^{\mathcal{N}} \ \mbox{ iff } \ 
\models^\mathcal{N}_w \square_1 \varphi \ \mbox{ iff } \ \termvalue{\varphi}^{\mathcal{N}} \in N_1(w)  \ \mbox{ iff } \ w \in S_\neg(\termvalue{\varphi}^{\mathcal{N}}).$$
That is, $S_\neg(\termvalue{\varphi}^{\mathcal{N}}) = \termvalue{\square_1\varphi}^{\mathcal{N}}$. Analogously, $S_\cons(\termvalue{\varphi}^{\mathcal{N}}) = \termvalue{\square_2\varphi}^{\mathcal{N}}$. From this, it is easy to prove by induction on the complexity of the formula $\varphi \in For(\Sigma)$ that $\termvalue{\varphi}_{\mathcal{M}}=\termvalue{\varphi^t}^{\mathcal{N}}$, where $\varphi^t$ is the formula over the signature $\Sigma_{bm}$ obtained from  $\varphi$ by replacing any ocurrence of the connectives $\neg$ and $\cons$ by the corresponding abbreviations, as indicated above.  Conversely, any neighborhood model $\matM=\langle W, S_\neg,S_\cons,d\rangle$ for \rmbc\ gives origin to a unique minimal model $\mathcal{N} = \langle W, N_1,N_2, d \rangle$ for  ${\bf E} {\oplus} {\bf E}$ such that $\termvalue{\varphi}_{\mathcal{M}}=\termvalue{\varphi^t}^{\mathcal{N}}$ for every formula $\varphi \in For(\Sigma)$. That is, the class of  minimal models for  ${\bf E} {\oplus} {\bf E}$  coincides (up to presentation) with the class of neighborhood models for \rmbc, and both classes validate the same formulas over the signature $\Sigma$ of \rmbc. From this,  Corollary~\ref{soundneighrmbc}, Theorem~\ref{BALFI-kripke} and Theorem~\ref{compleE+E} we show that \rmbc\ is definable within ${\bf E} {\oplus} {\bf E}$:

\begin{teom} \label{RmbC-ExE} The logic \rmbc\ is definable within ${\bf E} {\oplus} {\bf E}$, in the following sense: $\Gamma \vdash_{\rmbc} \varphi$ \ iff \ $\Gamma^t \vdash_{{\bf E} {\oplus} {\bf E}} \varphi^t$ for every $\Gamma \cup \{\varphi\} \subseteq For(\Sigma)$, where $\Gamma^t=\{\psi^t \ : \ \psi \in \Gamma\}$. 
\end{teom}

The main result  obtained in this section, namely Theorem~\ref{RmbC-ExE}, establishes an interesting relation between non-normal modal logics and paraconsistent logics. 

Connections between modalities and paraconsistency, or  proposals emphasizing a connection between  negations, denial, and impossibility, are well-known in the literature. \citet[p.~128]{sege82} suggests the possibility of studying the (unexplored at that time) modal notion of `$\varphi$ is non-necessary', namely $\sneg\square\varphi$. Along the same lines,  \citet{dosen:84} investigates modal operators corresponding to impossibility and non-necessity in systems analogous to  the modal system {\bf K} based on intuitionistic (instead of classical) logic. A system where impossibility is equivalent to intuitionistic negation is also investigated (proving soundness and completeness).
\citet{dosen:86} investigates negation as a  modal impossibility  operator added to a negationless logic. In turn, \cite{vakarelov:1989}\footnote{This is a summary of Vakarelov's PhD thesis ``Theory of negation in certain logical systems. Algebraic and semantical approach'', University of Warsaw, 1976.}
builds upon a distributive (lattice) logic to introduce a class of negations (generalizing intuitionistic negation) termed \textit{regular} negations by drawing upon the notion of theory consistency/compatibility. Some convenient semantical simplifications giving rise to the special subclasses of \textit{normal} and \textit{standard} negations are also discussed. In an analogous manner, by drawing upon the dual notion of theory completeness, Vakarelov introduces the dual class of \textit{co-regular} (as well as \textit{co-normal} and \textit{co-standard}) negations. Vakarelov also discusses paraconsistent and paracomplete negations and relates them to particular subclasses of regular and co-regular logics. Along similar lines, \citet{dunn2005negation} (drawing upon \citet{dunn1993star}) study two families of negative modal operators interpreted as impossibility and (dually) as non-necessity, each represented in an ordered structure termed ``(dual) kite of negations''.

More recently, \cite{bez:2002,bez:2005} proposes to  examine a paraconsistent negation defined in the modal system {\bf S5} as $\neg \varphi \defin \lozenge \sneg \varphi$  (which is equivalent to $\sneg\square\varphi$). Historically situated between the early works by Do\v{s}en and Vakarelov mentioned above and B\'eziau's proposal, a similar way of defining a paraconsistent negation inside a modal logic  has been regarded in~\citet{deAra:alv:guer:87}, where the authors propose a Kripke-style semantics for Sette's three-valued paraconsistent logic {\bf P1} based on Kripke frames for the modal logic {\bf T}. This result was improved in~\cite{car:lim:99}, by showing that  {\bf P1} can be interpreted in {\bf T} by means of Kripke frames  having at most two worlds.

Several authors have explored the possibility of defining such paraconsistent negation in other modal logics such as {\bf B} \citep{avr:16}, {\bf S4} \citep{CP17} and even weaker modal systems \citep{bue:car:17}. In such context,  \cite{mar:05b} proposes, besides the paraconsistent negation defined as above, the definition of a consistency connective within a modal system by means of the formula $\cons\varphi\defin \varphi \to \square\varphi$. In that paper it is shown  that any normal  modal logic in which the schema $\varphi \to \square\varphi$ is not valid gives origin to an \lfi\ in this way. Moreover, it is shown that it is also possible to start from a ``modal \lfi'', over the signature $\Sigma$ of \lfis, in which the paraconsistent negation and the consistency connective feature a Kripke-style semantics, defining the modal necessity operator by means of the formula $\square\varphi \defin \sneg\wneg\varphi$ (where $\sneg$ is the strong negation defined in terms of $\wneg$ and $\cons$ as in \mbc, recall Remark~\ref{obsRmbC-ext}~(2)). This shows that `reasonable' normal modal logics and \lfis\ with Kripke-style semantics are two faces of the same coin. 

Our Theorem~\ref{RmbC-ExE} partially extends this relationship to the realm of non-normal modal logics. The result we have obtained is partial, in the sense that the minimum bimodal non-normal modal logics gives origin to \rmbc, but the converse does not seem to be true. Namely, starting from \rmbc\ it is not obvious that the modalities $\square_1$ and $\square_2$ could be defined  by means of formulas in the signature $\Sigma$. This topic deserves further investigation.

\subsection{Neighborhood models for axiomatic extensions of \rmbc} \label{kripke-extens}

In this section $Ax$ will denote a set formed by one or more of the axiom schemas introduced in Definition~\ref{exte-mbc} with the exception of (\axca$_\#$) for $\# \in \{\land, \vee, \to\}$ (given the limitations to paraconsistency imposed to logic \rcila\ (recall Example~\ref{C1}).
Let $\mathbb{NM}(Ax)$ the class of neighborhood frames in which every schema in $Ax$ is valid. Define the consequence relation $\models_{\mathbb{NM}(Ax)}$ in the obvious way. By adapting the previous results it is easy to prove the following:

\begin{teom} {\em (Soundness and completeness of $\rmbc(Ax)$ w.r.t. $\mathbb{\mathbb{NM}}(Ax)$)}  \label{adeq-ext-Kripke} Let $\Gamma \cup \{\varphi\} \subseteq For(\Sigma)$. Then: $\Gamma \vdash_{\rmbc(Ax)} \varphi$ \ if and only if $\Gamma \models_{\mathbb{NM}(Ax)} \varphi$.
\end{teom}

The class of neighborhood frames that validates each of the axioms of $Ax$ can be easily characterized:
  
\begin{prop} \label{model-Kripke-axioms}
Let  \matF\ be a neighborhood frame for \rmbc.
Then:\\[1mm]
(1) $(\axwci)$ is valid in \matF\  \ iff \ $W \setminus (X \cap S_\neg(X)) \subseteq S_\cons(X)$, for every $X\subseteq W$;\\[1mm]
(2)  $(\axci)$ is valid in \matF\ \ iff \ $W \setminus (X \cap S_\neg(X)) \subseteq S_\cons(X)\setminus S_\neg((W \setminus (X \cap S_\neg(X))) \cap S_\cons(X))$, for every $X\subseteq W$;\\[1mm]
(3) $(\axcl)$ is valid in \matF\ \ iff \ $S_\neg(X \cap S_\neg(X)) \subseteq W \setminus (X \cap S_\neg(X)) \subseteq S_\cons(X)$, for every $X\subseteq W$;\\[1mm]
(4) $(\axcf)$ is valid in \matF\ \ iff \ $(X\setminus S_\neg(X)) \cup S_\neg(X\to S_\neg(X)) \subseteq X$, for every $X\subseteq W$;\\[1mm]
(5) $(\axce)$ is valid in \matF\ \ iff \ $X\subseteq (X\setminus S_\neg(X)) \cup S_\neg(X\to S_\neg(X))$, for every $X\subseteq W$.
\end{prop}

Recall the minimal bimodal logic ${\bf E} {\oplus} {\bf E}$  studied in Section~\ref{non-normal}. 
If $\varphi$ is a formula in $For(\Sigma_{bm})$ then ${\bf E} {\oplus} {\bf E}(\varphi)$ will denote the  extension of ${\bf E} {\oplus} {\bf E}$ by adding $\varphi$  as an axiom schema. Let $\mathcal{C'_M}(\varphi)$ be the class of structures (i.e., minimal models) $\mathcal{N}$ for ${\bf E} {\oplus} {\bf E}$ such that $\varphi$ is valid in $\mathcal{N}$  (as an axiom schema). Theorem~\ref{compleE+E} can be extended to prove that the logic ${\bf E} {\oplus} {\bf E}(\varphi)$ is sound and complete w.r.t. the semantics in  $\mathcal{C'_M}(\varphi)$. From this, and taking into account the representability of \rmbc\ within ${\bf E} {\oplus} {\bf E}$  (Theorem~\ref{RmbC-ExE}) and the equivalence between minimal models for  ${\bf E} {\oplus} {\bf E}$ and neighborhood models for \rmbc\ discussed right before Theorem~\ref{RmbC-ExE}, Proposition~\ref{model-Kripke-axioms} can be recast as follows:

\begin{coro} \label{coro-Kripke-axioms} \ \\
(1) $\rmbc(\axwci)$  is definable in ${\bf E} {\oplus} {\bf E}(\sneg(\varphi \land \square_1 \varphi) \to \square_2\varphi)$;\\[1mm]
(2) $\rmbc(\axci)$  is definable in ${\bf E} {\oplus} {\bf E}(\sneg(\varphi \land \square_1\varphi) \to	(\square_2\varphi \land \sneg\square_1(\sneg(\varphi \land \square_1\varphi) \land \square_2\varphi)))$
or, equivalenty, in
${\bf E} {\oplus} {\bf E}((\square_2\varphi \to\square_1(\sneg(\varphi \land \square_1\varphi) \land \square_2\varphi)) \to (\varphi \land \square_1\varphi))$;\\[1mm]
(3) $\rmbc(\axcl)$  is definable in 
${\bf E} {\oplus} {\bf E}((\square_1(\varphi \land \square_1\varphi) \to \sneg(\varphi \land \square_1\varphi)) \land (\sneg(\varphi \land \square_1\varphi) \to \square_2\varphi))$;\\[1mm]
(4) $\rmbc(\axcf)$ is definable in  
${\bf E} {\oplus} {\bf E}(((\varphi \land \sneg\square_1\varphi) \vee \square_1(\varphi \to \square_1\varphi)) \to \varphi)$;\\[1mm]
(5) $\rmbc(\axce)$ is definable in  
${\bf E} {\oplus} {\bf E}(\varphi \to ((\varphi \land \sneg\square_1\varphi) \vee \square_1(\varphi \to\square_1\varphi)))$.
\end{coro}

We have, in fact, that $\rmbc(\axwci)$ and its extension $\rmbc(\axci)$ are definable in \textbf{E}. Recall from Example~\ref{def-rmbcciw} that for them $\circ$ can be defined as a derived operator.

\section{Truth-preserving (or global) semantics} \label{truth-pres}

As  mentioned in Section~\ref{sect-rmbc}, the BALFI semantics for \rmbc, as well as its  neighborhood semantics presented in Section~\ref{sec:nbhd}, is degree-preserving instead of truth-preserving (using the terminology from~\cite{DegTru}). This requires adapting, in a coherent way, the usual definition of derivation from premises in a Hilbert calculus (recall Definition~\ref{deriv-RmbC}).
This is in fact the methodology adopted with most normal modal logics in which the semantics is  local, thus recovering the deduction metatheorem. But it is also possible to consider global (or truth-preserving) semantics, as it is usually done with algebraic semantics. This leads us to propose the logic $\rmbc^*$, which is defined by the same Hilbert calculus than the one for \rmbc, but now derivations from premises in  $\rmbc^*$  are defined as usual in Hilbert calculi.

\begin{defn} \label{rules1} The logic $\rmbc^*$ is defined by the same Hilbert calculus over signature $\Sigma$ than \rmbc, that is, by adding to  \mbc\ the inference rules (\Rn) and (\Rb).
\end{defn}

\begin{defn} [Derivations in  $\rmbc^*$] \label{der-global}
We say that a formula $\varphi$ is {\em derivable in $\rmbc^*$ from $\Gamma$}, and we write $\Gamma\vdash_{\rmbc^*}\varphi$, if there exists a finite sequence of formulas
$\varphi_1\ldots \varphi_n$ such that $\varphi_n$ is $\varphi$ and, for every $1 \leq i \leq n$, either  $\varphi_i$ is an instance of an axiom schema of \rmbc, or $\varphi_i \in \Gamma$, or  $\varphi_i$ is the consequence of some inference rule of \rmbc\  whose premises appear in the sequence $\varphi_1\ldots \varphi_{i-1}$.
\end{defn}

Now, the degree-preserving BALFI semantics for \rmbc\ given in Definition~\ref{degBALFI} is replaced by a truth-preserving consequence relation for  $\rmbc^*$. Notice however that, by virtue of the self-extensionality, this semantical notion will not require finitariness by definition: indeed, after proving soundness and completeness, finitariness of the semantical consequence relation (that is, the semantical compactness theorem for $\rmbc^*$) will be obtained as a corollary.

\begin{defn} [Truth-preserving BALFI semantics] \ \\ \label{truBALFI}
Let $\Gamma \cup\{\varphi\}$ be a set of formulas in $For(\Sigma)$. We say that $\varphi$ is a {\em truth-preserving} (or {\em global}) consequence of $\Gamma$ in the variety $\mathbb{BI}$ of BALFIs, denoted by $\Gamma \models^g_{\mathbb{BI}} \varphi$, if, for every BALFI \B\ and every valuation $v$ over it,  $v(\varphi)=1$ whenever $v(\gamma)=1$ for every $\gamma \in \Gamma$.
\end{defn}

\begin{teom} {\em (Soundness and completeness of $\rmbc^*$ w.r.t. truth-preserving semantics)}  \label{SCRmbC*}
For every $\Gamma \cup\{\varphi\} \subseteq For(\Sigma)$:
 $\Gamma\vdash_{\rmbc^*}\varphi$ \ iff \ $\Gamma \models^g_{\mathbb{BI}} \varphi$.
\end{teom}
\begin{proof}
(Soundness) The proof is a trivial adaptation of the argument given for \rmbc\ in the proof of Theorem~\ref{sound-BALFI}.\\
(Completeness) It can be obtained by a slight adaptation of the proof of Theorem~\ref{comple-BALFI} given above. Thus, suppose that $\Gamma  \nvdash_{\rmbc^*} \varphi$.
Define the following relation on $For(\Sigma)$: $\alpha \equiv_\Gamma \beta$ iff $\Gamma\vdash_{\rmbc^*} \alpha \sse \beta$. By similar arguments to the ones given in the proof of Theorem ~\ref{comple-BALFI}, it is easy to see that $\equiv_\Gamma$ is an equivalence relation. Moreover, it is a congruence on $For(\Sigma)$: indeed, the only connectives that deserve special attention are $\neg$ and $\circ$. However, since  rules (\Rn) and (\Rb) can be used without restrictions in derivations in $\rmbc^*$, it follows that,  for $\# \in \{\neg,\cons\}$,  $\#\alpha \equiv_\Gamma \#\beta$ provided that $\alpha \equiv_\Gamma \beta$. This being so, the quotient set $A_{can}^\Gamma\defin For(\Sigma)/_{\equiv_\Gamma}$ is the carrier of a  BALFI, denoted by $\mathcal{B}_{can}^\Gamma$, with operations given as follows (where $[\alpha]_\Gamma$ denotes the equivalence class of $\alpha$ w.r.t. $\equiv_\Gamma$): $[\alpha]_\Gamma \,\#\, [\beta]_\Gamma \defin [\alpha \# \beta]_\Gamma$, for $\# \in \{\land,\lor,\imp\}$;  $\#[\alpha]_\Gamma \defin [\#\alpha]_\Gamma$, for $\# \in \{\neg,\cons\}$; $0 \defin [\alpha \wedge \neg \alpha \wedge \cons \alpha]_\Gamma$ and $1 \defin [\alpha \vee \neg \alpha]_\Gamma$ (for any $\alpha$). Now, let $v_{can}^\Gamma:For(\Sigma) \to A_{can}^\Gamma$ given by $v_{can}^\Gamma(\alpha)=[\alpha]_\Gamma$. Clearly $v_{can}^\Gamma$ is a valuation over  $\mathcal{B}_{can}^\Gamma$ such that $v_{can}^\Gamma(\alpha)=1$ iff $\Gamma \vdash_{\rmbc^*} \alpha$. From this,  $v_{can}^\Gamma(\alpha)=1$ for every $\alpha \in \Gamma$, while $v_{can}^\Gamma(\varphi)\neq 1$. This shows that $\Gamma  \not\models^g_{\mathbb{BI}} \varphi$.
\end{proof}

\begin{cor} [Semantical compactness of $\rmbc^*$] If $\Gamma  \models^g_{\mathbb{BI}} \varphi$ then there exists a finite subset $\Gamma_0$ of $\Gamma$ such that $\Gamma_0  \models^g_{\mathbb{BI}} \varphi$.
\end{cor}
\begin{proof}
It is an immediate consequence of Theorem~\ref{SCRmbC*} and the fact that the consequence relation  $\vdash_{\rmbc^*}$ is finitary, by definition.
\end{proof}

\begin{rem}
The definition of truth-preserving semantics restricts the number of paraconsistent models for   $\rmbc^*$.  Indeed, let $p$ and $q$ be two different propositional variables. In order to show that $p,\neg p \not \models^g_{\mathbb{BI}} q$, there must be a   BALFI \B\ and a valuation $v$ over \B\ such that $v(p)=v(\neg p)=1$ but $v(q)\neq 1$. That is, \B\ must be such that $\neg 1 = 1$. Since $\neg 0 = 1$, it follows that $\neg \neg 0 = \neg 1 = 1 \not\leq 0$ in \B. This shows that there is no paraconsistent extension of   $\rmbc^*$ which satisfies axiom $(\axcf)$. In particular, there is no paraconsistent extension of $\rmbc^*$ satisfying axioms  $(\axcf)$ and  $(\axci)$. Thus, the open problems solved in Examples~\ref{exem-bC} and~\ref{exem-Ci} have a negative answer in this setting. This shows that the truth-preserving approach is much more restricted than the degree-preserving approach in terms of paraconsistency. 

In any case, there are {\em still} paraconsistent BALFIs for the truth-preserving logic  $\rmbc^*$ (namely, the ones such that $\neg 1=1$). The situation is quite different in the realm of fuzzy logics: in~\cite{CEG} and \cite{EEFGN}, among  others, the degree-preserving companion of several fuzzy logics has been studied, showing that their usual truth-preserving consequence relations are never paraconsistent. 
\end{rem}

The distinction between  local and global reasoning has been studied by A. Sernadas and his collaborators (for a brief exposition see, for instance, \cite{car:etal:2008}, Section~2.3 in Chapter~2). From the proof-theoretical perspective, the Hilbert calculi (called {\em Hilbert calculi with careful reasoning} in~\citet[Definition~2.3.1]{car:etal:2008}) are of the form $H= \langle \Theta, R_g,R_l\rangle$ where $\Theta$ is a propositional signature and $R_g \cup R_l$ is a set of inference rules such that $R_l \subseteq R_g$ and no element of $R_g\setminus R_l$ is an axiom schema. Elements of $R_g$ and $R_l$ are called {\em global} and {\em local} inference rules, respectively. Given $\Gamma \cup \{\varphi\} \subseteq For(\Theta)$, $\varphi$ is {\em globally derivable} from $\Gamma$ in $H$, written $\Gamma \vdash_H^g \varphi$, if $\varphi$ is derivable from $\Gamma$ in the Hilbert calculus $\langle \Theta, R_g\rangle$ by using the standard definition (see Definition~\ref{der-global}). On the other hand, in local derivations, besides using the local rules and the premises, global rules can be used provided that the premises are (global) theorems. In formal terms, $\varphi$ is {\em locally derivable} from $\Gamma$ in $H$, written $\Gamma \vdash_H^l \varphi$, if  there exists a finite sequence of formulas
$\varphi_1\ldots \varphi_n$ such that $\varphi_n$ is $\varphi$ and, for every $1 \leq i \leq n$, either $\varphi_i \in \Gamma$, or $\vdash_H^g\varphi_i$, or  $\varphi_i$ is the consequence of some inference rule of $R_l$  whose premises appear in the sequence $\varphi_1\ldots \varphi_{i-1}$ (observe that this includes the case where  $\varphi_i$ is an instance of an axiom in $R_l$). Obviously, local derivations are global derivations, but the converse does not  necessarily hold: for instance, $\varphi \vdash_H^g \square\varphi$ always hold in a Hilbert calculus $H$ for a  normal modal logic (say, {\bf K}), while $\varphi \vdash_H^l  \square\varphi$ holds iff $\varphi$ is a theorem.\footnote{This is clearer in terms of Kripke semantics: if $\varphi$ is a contingent formula, there is a Kripke model $M$ which satisfies $\varphi$ in a world $w$, but there may exist another world $w'$ related to $w$ in which $M$ does not satisfy $\varphi$; hence, $M$ does not satisfy $\square\varphi$ in $w$.} Clearly, local and global theorems coincide. 

For instance,  a Hilbert calculus for a (normal) modal logic  typically contains, as local inference rules, (\MP) and the axiom schemas, while the set of global rules is $(R_l)$ plus the \textit{necessitation rule}. As we have seen in Section~\ref{non-normal}, for the minimal non-normal modal logic \textbf{E} we have a similar case, with the replacement rule for $\square$ being employed instead of the \textit{necessitation rule}. In this case, the deduction metatheorem only holds for local derivations. Note that, by definition, derivations in $\rmbc^*$ lie in the scope of global derivations, while derivations in \rmbc\ are local derivations. Hence, the extension of \mbc\ with replacement can be recast as a Hilbert calculus  with careful reasoning  $\rmbc^ + = \langle \Sigma,R_g,R_l\rangle$  such that $R_l$ contains the axiom schemas of \mbc\ plus (\MP), and $R_g$ contains, besides this, the rules (\Rn) and (\Rb). Of course, the same can be done with the axiomatic extensions of \mbc\ (and so of \rmbc) considered in Section~\ref{Sexte-mbc}.

At the semantical level, local derivations correspond to degree-preserving semantics w.r.t. a given class $\mathbb{M}$ of algebras, while global derivations correspond to truth-preserving semantics w.r.t. the class $\mathbb{M}$.

The  presentation of \lfis\ with replacement as Hilbert calculi with careful reasoning (as the case of $\rmbc^ +$)  can be useful in order to combine these logics with (standard) normal modal logics by algebraic fibring: in this case, completeness of the fibring of the corresponding Hilbert calculi w.r.t. a semantics given by classes of suitable expansions of Boolean algebras would be immediate, according to the results stated in~\citet[Chapter~2]{car:etal:2008}. By considering, as done in~\cite{zan:acs:css:99}, classes $\mathbb{M}$ of powerset algebras (i.e., with carrier of the form  $\wp(W)$ for a non-empty set $W$) induced by Kripke models (which can be generalized to neighborhood models), then the fibring of, say, $\rmbc^ +$ with a given modal logic would simply be a minimal bimodal logic ${\bf E} {\oplus} {\bf E}$ with an additional modality $\square$, from which we derive the following modalities: $\lozenge\varphi \defin \sneg\square\sneg\varphi$, $\neg \varphi \defin \varphi \to \square_1\varphi$, and $\cons\varphi \defin \sneg(\varphi \land \square_1\varphi) \land \square_2\varphi$.
This opens interesting opportunities for future research.

\section{Extension to first-order logics}  \label{extension-FOL}

The next step is extending \rmbc, as well as its axiomatic extensions analyzed above, to first-order languages.
We begin by extending the notion of a propositional signature to the first-order case.

\begin{defn}\label{fosig}
	Let $\Sigma$ be a propositional signature. A first-order extension of $\Sigma$ is a first-order signature $\Omega = \langle \Sigma,\,\mathcal{C},\,\{\mathcal{F}_n\}_{n\geq 1},\,\{\mathcal{P}_n\}_{n\geq 1} \rangle$ where:\\ 
	
	$\begin{array}{ll}
	- &  \mbox{$\mathcal{C}$ is a denumerable set of individual constants};\\[1mm] 
	- & \mbox{each $\mathcal{F}_n$ is a denumerable set of function symbols of arity $n$},\\[1mm]
	- & \mbox{each $\mathcal{P}_n$ is a denumerable set of predicate symbols of arity $n$}.
	\end{array}$
\end{defn}

The set $\mathcal{C}$ as well as any of $\mathcal{F}_n$ and $\mathcal{P}_n$ may be empty. However, $\bigcup_{n\geq 1}\mathcal{P}_n$ must be nonempty.
	In the sequel, we work with a fixed denumerable set $Var=\{v_1,v_2,\ldots\}$ of individual variables. The first-order language over $\Omega$ is defined as follows. The set of terms \tert\ (resp.~closed terms \ctert) is freely generated by $Var\,\cup\,\mathcal{C}$ (resp.~$\mathcal{C}$) employing the elements of $\bigcup_{n\geq 1}\mathcal{F}_n$ as connectives. The set of atomic formulas \afor\ is composed of expressions of the form $P(\tau_1,\ldots,\tau_n)$ where $P \in \mathcal{P}_n$ and $\tau_i \in $ \tert\ for $1\leq i \leq n$. Finally, the set of formulas \fort\ is freely generated by \afor\ employing $\Sigma \cup \{\forall v\}_{v \in  Var} \cup \{\exists v\}_{v \in  Var}$ as connectives, where the latter two are $Var$-indexed sets of quantifier symbols which we treat as having arity $1$.

The set of closed formulas (or sentences) will be denoted by \sent. The formula obtained from a given formula $\varphi$ by substituting every free occurrence of a variable $x$ by a term $t$ will be denoted by $\varphi[x/t]$.

\begin{defn} Let $\Omega$ be a first-order signature extending $\Sigma$ from Definition~\ref{signa}. The logic \rqmbc\ is obtained from  \rmbc\ by adding the following axioms and rules:\\[1mm]
	
	{\bf Axiom Schemas:}\\
	
	$\begin{array}{ll}
	{\bf (Ax\exists)} & \varphi[x/t]\to\exists x\varphi, \ \mbox{ if $t$ is a term free for $x$ in $\varphi$}\\[3mm]
	{\bf (Ax\forall)} & \forall x\varphi\to\varphi[x/t], \ \mbox{ if $t$ is a term free for $x$ in $\varphi$}
	\end{array}$\\
	
	\
	
	{\bf Inference rules:}\\
	
	$\begin{array}{ll}
	{\bf (\exists\mbox{\bf -In})} & \dfrac{\varphi\to\psi}{\exists x\varphi\to\psi}, \ \mbox{  where $x$ does not occur free in $\psi$}\\[4mm]
	{\bf (\forall\mbox{\bf -In})} & \dfrac{\varphi\to\psi}{\varphi\to\forall x\psi}, \ \mbox{ where $x$ does not occur free in $\varphi$}
	\end{array}$
\end{defn}\

The consequence relation of \rqmbc, suitably extending the one for \rmbc\ (recall Definition~\ref{deriv-RmbC}) with the axioms and rules above, will be denoted by $\vdash_{\rqmbc}$. 
Recall that a  Hilbert calculus  with careful reasoning for \rmbc\ called   $\rmbc^+$ was defined at the end of Section~\ref{truth-pres}. This can be extended to the first-order case by considering the  Hilbert calculus  with careful reasoning $\rqmbc^+$ over a given first-order signature $\Omega$, such that $R_l$ contains (\MP) and the axiom schemas of {\mbc} plus ${\bf (Ax\exists)}$ and ${\bf (Ax\forall)}$. In turn, $R_g$ contains, besides the rules $R_l$, the (global) rules $(\Rn)$, $(\Rb)$, $(\exists\mbox{\bf -In}$) and $(\forall\mbox{\bf -In})$ (over $\Omega$).
The property of replacement for {\rqmbc} follows directly from the propositional case (Theorem~\ref{IPE-rmbc}), observing that quantifiers (seen as unary operators) clearly preserve logical equivalence.

\begin{teom} \label{IPE-rqmbc} Let {\bf L} be the logic \rqmbc\ or any axiomatic extension of it over the signature $\Omega$, in which the derivation from premises is defined as in \rqmbc. Then, {\bf L} satisfies the replacement property.
\end{teom}

Given that, as in the case of \rmbc, \rqmbc\ uses local reasoning, it satisfies the deduction metatheorem without any restrictions. This is different from what happens with \qmbc, a first-order version of \mbc\ proposed in~\cite{car:etal:14}, where this metatheorem holds with the same restrictions as in first-order classical logic.\footnote{It is worth mentioning that the only difference between {\qmbc} and {\rqmbc} is that the latter contains the inference rules $(\Rn)$ and $(\Rb)$, which are not present in the former (besides the different notions of derivation from premises adopted in {\qmbc} and in {\rqmbc}). In addition, da Costa's axioms included in \qmbc, which state the equivalence between two formulas  such that one is a variant of the other (that is, they coincide up to void quantifiers and renaming of bounded variables), is no longer necessary in \rqmbc.}

\begin{teom}[Deduction Metatheorem for \rqmbc] \ \\
	$\Gamma, \varphi \vdash_{\rqmbc}  \psi$ \ if and only if $\Gamma \vdash_{\rqmbc} \varphi \to \psi$.
\end{teom}

\subsection{BALFI semantics for \rqmbc} \label{balfifol}

In~\cite{CFG19} a  semantics of first-order structures, based on  \textit{swap structures} over complete Boolean algebras, was obtained for \qmbc. Since \rqmbc\ is self-extensional, that semantics can be drastically simplified, and so the non-deterministic swap structures will be replaced by BALFIs, which are ordinary algebras. The semantics will be defined in an analogous fashion as we did for the propositional case, where valuations can be considered as homomorphisms between the algebra of formulas and BALFIs.
	In the sequel, only BALFIs over	\textit{complete} Boolean algebras will be considered. This requirement is justified since quantifiers will be interpreted by means of (infinite) infima and suprema in Boolean algebras, hence their existence must be always guaranteed.

\begin{defn} A Boolean algebra \B\ is {\em complete} if there exist in \B\ both the infimum $\bigwedge X$ and the supremum $\bigvee X$ of every subset $X$ of its carrier. A  {\em complete  BALFI} is a  BALFI such that its  reduct  to $\Sigma_{BA}$ is a complete Boolean algebra.
\end{defn}

\begin{defn} Let $\Omega$ be a first-order signature. Moreover, let \B\ be a complete BALFI with carrier $A$ and $U$ be some nonempty set (called universe or domain). A  (first-order) {\em structure} (or a {\em\rqmbc-structure}) over $\Omega$ is a tuple $\mathfrak{A} = \langle U, \B, I_{\mathfrak{A}} \rangle$ such that $I_{\mathfrak{A}}$ is an interpretation function which assigns:\\
	
	$\begin{array}{ll}
	- & \mbox{ an element $I_\mathfrak{A}(c)$ of $U$ to each individual constant $c \in \mathcal{C}$};\\[1mm] 
	- & \mbox{ a function $I_\mathfrak{A}(f): U^n \to U$ to each function symbol $f$ of arity $n$};\\[1mm] 
	- & \mbox{ a function  $I_\mathfrak{A}(P): U^n \to A$ to each predicate symbol $P$ of arity $n$}.
	\end{array}$
\end{defn}


From now on, we will write $c^\mathfrak{A}$, $f^\mathfrak{A}$ and $P^\mathfrak{A}$ instead of $I_\mathfrak{A}(c)$, $I_\mathfrak{A}(f)$ and $I_\mathfrak{A}(P)$ to denote the interpretation of an individual constant symbol $c$, a function symbol $f$ and a predicate symbol $P$, respectively.

\begin{defn} Given a structure $\mathfrak{A}=\langle U, \B, I_{\mathfrak{A}} \rangle$ over $\Omega$,  an {\em assignment} over  $\mathfrak{A}$ is any function $\mu: Var \to U$.  The set of assignments over $\mathfrak{A}$ will be denoted by $\textsc{A}(\mathfrak{A})$. For every $a \in U$ and every assignment $\mu$, let $\mu^{x}_{a}$ be the assignment such that $\mu^{x}_{a}(x)=a$ and $\mu^{x}_{a}(y)=\mu(y)$ if $y \neq x$.
\end{defn}

\begin{defn}[Term denotation] \label{term} 
	Given  a structure $\mathfrak{A}$ over $\Omega$, and given an assignment $\mu$ over $\mathfrak{A}$, we define recursively, for each term $t$,  an element $\termvalue{t}^\mathfrak{A}_\mu$ in $U$ as follows:\\
	
	$\begin{array}{ll}
	- & \termvalue{c}^\mathfrak{A}_\mu = c^\mathfrak{A} \ \mbox{ if $c$ is an individual constant};\\[1mm]
	- & \termvalue{x}^\mathfrak{A}_\mu = \mu(x) \ \mbox{ if $x$ is a variable};\\[1mm]
	- & \termvalue{f(t_1,\ldots,t_n)}^\mathfrak{A}_\mu = f^\mathfrak{A}(\termvalue{t_1}^\mathfrak{A}_\mu,\ldots,\termvalue{t_n}^\mathfrak{A}_\mu) \ \mbox{ if $f$ is a function symbol of arity $n$}\\[1mm]
	& \mbox{ and $t_1,\ldots,t_n$ are terms}.
	\end{array}$
\end{defn}

	\begin{defn}[\rqmbc\ interpretation maps]~\label{val}
		Let \B\ be a complete BALFI and $\mathfrak{A}=\langle U, \B, I_{\mathfrak{A}} \rangle$ a structure over $\Omega$. The {\em interpretation map} for \rqmbc\ over $\mathfrak{A}$ is a function $\termvalue{\cdot}^{\mathfrak{A}}_{\mu}:For_1(\Omega)\times \textsc{A}(\mathfrak{A}) \to A$ (for a given assigment $\mu$ over $\mathfrak{A}$) satisfying the following clauses:\\
		
		$\begin{array}{ll}
		(i) & \termvalue{P(t_1,\ldots,t_n)}^{\mathfrak{A}}_\mu = P^{\mathfrak{A}}(\termvalue{t_1}^{\mathfrak{A}}_\mu,\ldots,\termvalue{t_n}^{\mathfrak{A}}_\mu), \ \mbox{ for $P(t_1,\ldots,t_n)$ atomic}; \\[2mm]
		(ii) & \termvalue{\#\varphi}^{\mathfrak{A}}_\mu = \# \termvalue{\varphi}^{\mathfrak{A}}_\mu, \ \mbox{ for every $\#\in \{\neg, \cons\}$};\\[2mm]
		(iii) & \termvalue{\varphi \,\#\, \psi}^{\mathfrak{A}}_\mu = \termvalue{\varphi}^{\mathfrak{A}}_\mu \,\#\, \termvalue{\psi}^{\mathfrak{A}}_\mu, \ \mbox{ for every $\#\in \{\wedge,\vee, \to\}$};\\[2mm]
		(iv) & \termvalue{\forall x\varphi}^{\mathfrak{A}}_\mu = \bigwedge\{\termvalue{\varphi}^{\mathfrak{A}}_{\nu} \;: \; {\nu = \mu^x_a \;;\; a \in U}\};\\[2mm]
		(v) & \termvalue{\exists x\varphi}^{\mathfrak{A}}_\mu = \bigvee\{\termvalue{\varphi}^{\mathfrak{A}}_{\nu} \;: \; {\nu = \mu^x_a \;;\; a \in U}\}.
		\end{array}$
	\end{defn}

\begin{rems} \ \\ 
	(1) If $t$ is a closed term we can write $\termvalue{t}^\mathfrak{A}$ instead of $\termvalue{t}^\mathfrak{A}_\mu$, for any   assignment $\mu$, since it does not depend on $\mu$. Analogously, we can write $\termvalue{\varphi}^\mathfrak{A}$ instead of $\termvalue{\varphi}^\mathfrak{A}_\mu$ for a closed formula $\varphi$.\\
	(2) The last two items can equivalently be written as:\\
		$$(iv^*)\;\;\termvalue{\forall x\varphi}^{\mathfrak{A}}_\mu = \bigwedge_{a \in U} \termvalue{\varphi}^{\mathfrak{A}}_{\mu^x_a} \text{\;\;\;\;\;\;and\;\;\;\;\;\;}
		(v^*)\;\;\termvalue{\exists x\varphi}^{\mathfrak{A}}_\mu = \bigvee_{a \in U} \termvalue{\varphi}^{\mathfrak{A}}_{\mu^x_a}$$
		If $x$ does not appear free in $\varphi$ we have that
		$\termvalue{\forall x\varphi}^{\mathfrak{A}}_\mu = \termvalue{\varphi}^{\mathfrak{A}}_\mu$ and
		$\termvalue{\exists x\varphi}^{\mathfrak{A}}_\mu = \termvalue{\varphi}^{\mathfrak{A}}_\mu$.
\end{rems}

\begin{defn} \label{consrel0} Let  $\mathfrak{A}$ be  a \rqmbc-structure over $\Omega$. A formula $\varphi$ in $For_1(\Omega)$, $\varphi$ is said to be   {\em valid in $\mathfrak{A}$}, denoted by $\models_{\mathfrak{A}}\varphi$, if  $\termvalue{\varphi}^{\mathfrak{A}}_{\mu}=1$,  for every assignment $\mu$. We say that $\varphi$ is  {\em valid in \rqmbc}, denoted by $\models_{\rqmbc}\varphi$, if $\models_{\mathfrak{A}}\varphi$, for every  $\mathfrak{A}$.
\end{defn}

\begin{defn} [First-order degree-preserving BALFI semantics] \label{consrel} Let $\Gamma \cup \{\varphi\}$ be a set of formulas in $For_1(\Omega)$. Then  $\varphi$ is said to be  a {\em semantical consequence of $\Gamma$ in \rqmbc\ w.r.t. BALFIs}, denoted by $\Gamma\models_{\rqmbc}\varphi$, if either $\varphi$ is valid in \rqmbc, or  there exists a finite, non-empty subset $\{\gamma_1,\ldots,\gamma_n\}$ of $\Gamma$ such that the formula $(\gamma_1 \wedge \ldots \wedge \gamma_n) \to \varphi$ is  valid in  \rqmbc.
\end{defn}

In order to prove the soundness of \rqmbc\ w.r.t. BALFI semantics, it is necessary to state an important result:

\begin{teom}[Substitution Lemma] \label{substlem}
		Let \B\ be a complete BALFI,   $\mathfrak{A}$ a structure over $\Omega$, and $\mu$ an assignment over $\mathfrak{A}$. If $t$ is a term free for $z$ in $\varphi$ and $b = \termvalue{t}^{\mathfrak{A}}_{\mu}$, then $\termvalue{\varphi[z/t]}^{\mathfrak{A}}_{\mu} = \termvalue{\varphi}^{\mathfrak{A}}_{\mu^z_b}$.
\end{teom}
\begin{proof}
	The result is easily proved by induction on the complexity of $\varphi$.
\end{proof}

\begin{teom} [Soundness of \rqmbc\ w.r.t. BALFIs] \ \\ \label{adeq-rqmnc-BALFI}
	For every set $\Gamma \cup\{ \varphi\}  \subseteq For_1(\Omega)$:   $\Gamma \vdash_\rqmbc \varphi$ implies that $\Gamma \models_\rqmbc \varphi$.
\end{teom}
\begin{proof} It will be proven by extending the proof of soundness of \rmbc\ w.r.t. BALFI semantics (Theorem~\ref{sound-BALFI}). Thus, the only cases required to be analyzed are the new axioms and inference rules. By the very definitions, and taking into account Theorem~\ref{substlem}, it is immediate to see that axioms ${\bf (Ax\exists)}$ and ${\bf (Ax\forall)}$ are valid in any $\mathfrak{A}$. With respect to ${\bf (\exists\mbox{\bf -In})}$, suppose that $\alpha \to \beta$ is valid in a given $\mathfrak{A}$, where the variable $x$ does not occur free in $\beta$. Then   $\termvalue{\alpha}^{\mathfrak{A}}_{\mu} \leq  \termvalue{\beta}^{\mathfrak{A}}_{\mu}$ for every assignment $\mu$. In particular, for every $a \in U$, $\termvalue{\alpha}^{\mathfrak{A}}_{\mu^{x}_{a}} \leq  \termvalue{\beta}^{\mathfrak{A}}_{\mu^{x}_{a}} = \termvalue{\beta}^{\mathfrak{A}}_{\mu}$, since $x$ is not free in $\beta$. But then: $\termvalue{\exists x \alpha}^{\mathfrak{A}}_{\mu} = \bigvee_{a \in U} \termvalue{\alpha}^{\mathfrak{A}}_{\mu^{x}_{a}} \leq \termvalue{\beta}^{\mathfrak{A}}_{\mu}$. Hence, $\exists x\alpha \to \beta$ is valid in $\mathfrak{A}$. The case for ${\bf (\forall\mbox{\bf -In})}$  is proved analogously.
\end{proof}

\subsection{Completeness of \rqmbc\ w.r.t. BALFI semantics} \label{complete}

Along this section, in which the  completeness of \rqmbc\ w.r.t. BALFI semantics will be shown, $\Omega$ will denote a fixed first-order signature, over which the consequence relations $\vdash_{\rqmbc}$ and $\models_\rqmbc$ are defined. In turn, $\Omega_{C}$ will denote the signature obtained from $\Omega$ by adding an infinite and denumerable set $C$ of new individual constants. The syntactical consequence relation of \rqmbc\ over the signature $\Omega_{C}$ will be denoted by  $\vdash^{C}_{\rqmbc}$.

\begin{prop}  \label{constants}
If  $\Gamma \cup \{\varphi\}\in For_1(\Omega)$, then: $\Gamma \vdash^{C}_{\rqmbc} \varphi$ \ iff \ $\Gamma\vdash_{\rqmbc}\varphi$.
\end{prop}
\begin{proof}
Observe that it is enough to prove the case for $\Gamma=\emptyset$, since the general case follows easily from this.
The `if' part is an immediate consequence of the definition of $\vdash^{C}_{\rqmbc}$. Now, assume that $\vdash^{C}_{\rqmbc} \varphi$, for $\varphi \in For_1(\Omega)$,  and let $d$ be a derivation   $\varphi_1\ldots \varphi_n$ of $\varphi$ in \rqmbc\ over the signature $\Omega_C$ (hence $\varphi_n=\varphi$). Since the derivation $d$ is finite and the length of the formulas is also finite, the set of constants in $C$ occurring in $d$ is contained in $\{c_1,\ldots,c_m\}$ for some $m \geq 1$. Let $x_1,\ldots,x_m$ be $m$ different fresh variables, that is, $x_i\neq x_j$ if $i\neq j$, and no $x_i$ occurs (free or bounded) in the formulas of the derivation $d$. Let 
$\varphi'_i$ be the formula obtained from $\varphi_i$ by substituting uniformly the constant $c_j$ by the variable $x_j$, for $1 \leq i \leq n$ and $1 \leq j \leq m$. Observe that $\varphi'_m=\varphi$. Then, it is clear that the sequence $d'$ given by  $\varphi'_1\ldots \varphi'_n$ is a derivation of $\varphi$ in \rqmbc\ over the signature $\Omega$. From  this, $\vdash_{\rqmbc}\varphi$.
\end{proof}

By an easy adaptation of the first part of the proof of Theorem~\ref{comple-BALFI}  the following can be proved:

\begin{prop} \label{AC} Let ${\equiv^{C}} \subseteq For_1(\Omega_C)^{2}$ be a relation in $For_1(\Omega_C)$ defined as follows:  $\alpha \equiv^C \beta$ iff $\vdash_{\rqmbc}^C \alpha \leftrightarrow\beta$. Then, $\equiv^{C}$ is a congruence over $For_1(\Omega_C)$ (seen as an algebra over $\Sigma$) such that the operations  $[\alpha]^{C} \# [\beta]^{C} \defin [\alpha \# \beta]^{C}$ for any $\# \in \{\wedge, \vee, \to\}$, $\# [\alpha]^{C}  \defin [\# \alpha]^{C}$ for any $\# \in \{\neg, \circ\}$, $0^{C} \defin [\alpha \wedge (\neg \alpha \wedge \circ \alpha)]^{C}$ and  $1^{C} \defin [\alpha \vee \neg \alpha]^{C}$ (where $[\alpha]^{C}$ denotes the equivalence class of $\alpha$ w.r.t. $\equiv^{C}$) define a  BALFI $\mathcal{B}_C^0$ with carrier $A_C \defin For_1(\Omega_C)/_{\equiv^C}$.
\end{prop}

The Boolean algebra  underlying the BALFI $\mathcal{B}_C^0$ will be denoted by  $\A_C$.

The construction of the canonical model for \rqmbc\  requires a complete BALFI, hence  the Boolean algebra $\mathcal{A}_C$ must be completed. Recall\footnote{See, for instance, \citet[Chapter~25]{gi:ha:09}.} that a Boolean algebra $\A'$ is a {\em completion} of a Boolean algebra \A\ if:~(1) $\A'$ is complete, and~(2) $\A'$ includes \A\ as a dense subalgebra (that is: every element in $A'$ is the supremum, in $\A'$, of some subset of $A$). From this, $\A'$ preserves all the existing infima and suprema in \A. In formal terms: there exists a monomorphism of Boolean algebras (therefore an injective mapping) $\ast:\A \to \A'$ such that $\ast(\bigvee_\A X)= \bigvee_{\A'} \ast[X]$  for every $X\subseteq A$ such that the supremum $\bigvee_\A X$ exists, where $\ast[X]=\{\ast(a) \ : \ a \in X\}$. Analogously, $\ast(\bigwedge_\A X)= \bigwedge_{\A'} \ast[X]$  for every $X\subseteq A$ such that the infimum $\bigwedge_\A X$ exists. By well-known results obtained  independently in~\cite{macneille}  and~\cite{tarski}, it follows that every Boolean algebra has a completion; moreover, the completion is unique up to isomorphisms. Based on this, let $\textsf{C}\A_C$ be the completion of $\A_C$ and let  $\ast:\A_C \to \textsf{C}\A_C$ be the associated monomorphism. Recalling from Proposition~\ref{AC} that $\A_C$ is the reduct of the BALFI $\mathcal{B}_C^0$, we can define the following:

\begin{defn}\label{nma}
 Let $\textsf{C}\A_C$ be the complete Boolean algebra defined as above. The canonical BALFI for \rqmbc\ over $\Omega_C$, denoted by  $\mathcal{B}_C$, is obtained from $\textsf{C}\A_C$ by adding the unary operators $\neg_\ast$ and $\cons_\ast$ defined as  follows: $\neg_\ast b = \ast(\neg a)$ if $b=\ast(a)$, and $\neg_\ast b = \sneg  b$ if $b \notin \ast[A_C]$ (where $\sneg$ is the Boolean complement in $\textsf{C}\A_C$ and $\neg$ is the operation in $\mathcal{B}_C^0$); $\cons_\ast b = \ast(\cons a)$ if $b=\ast(a)$, and $\cons_\ast b = 1$ if $b \notin \ast[A_C]$ (where $\ast$ is the operation in $\mathcal{B}_C^0$).\footnote{Observe that the definition of the operations $\neg_\ast$ and $\cons_\ast$ outside the image of $*$ meet the requirements of BALFIs. Of course other options,  coherent  with the definition of BALFIs, could be possible for such values outside $\ast[A_C]$, still defining a BALFI structure over $\textsf{C}\A_C$.}
\end{defn}

\begin{prop} \label{balfi-well}
The operations over $\mathcal{B}_C$ are well-defined, and $\mathcal{B}_C$ is a complete BALFI such that $\ast([\alpha]^C)=1 \ \mbox{ iff } \ \vdash_{\rqmbc}^C \alpha$.
\end{prop}
\begin{proof}
Since $\ast[A_C]$ is a subalgebra of  $\textsf{C}\A_C$,  $b \notin \ast[A_C]$ iff $\sneg b \notin \ast[A_C]$. On the other hand, $\ast$ is injective. This shows that $\neg_\ast$ and $\cons_\ast$ are well-defined. The rest of the proof is obvious from the definitions. 
\end{proof}

\begin{defn} (Canonical Structure) \label{str}
Let $\mathcal{B}_C$ be as in Definition~\ref{nma}, and let $U = CTer(\Omega_C)$. The {\em canonical structure induced by $C$} is the  structure  $\mathfrak{A}_C = \langle U, \mathcal{B}_C, I_{\mathfrak{A}_C} \rangle$ over $\Omega_C$ such that:\\

$\begin{array}{ll}
- & c^{\mathfrak{A}_C} = c, \ \mbox{ for each individual constant $c$};\\[2mm]
- & f^{\mathfrak{A}_C}: U^n \to U \ \mbox{ is such that } \ f^{\mathfrak{A}_C} (t_1, \ldots, t_n) = f(t_1, \ldots, t_n), \ \mbox{ for each}\\[2mm] 
& \mbox{$n$-ary function symbol $f$};\\[2mm]
- & P^{\mathfrak{A}_C}(t_1, \ldots, t_n) = \ast([P(t_1, \ldots, t_n)]^C), \ \mbox{ for each $n$-ary predicate symbol $P$}.
\end{array}$
\end{defn}

\begin{lemma} \label{quantOK}
Let  $\A_C$ be the Boolean algebra obtained as in Proposition~\ref{AC}, and let $U = CTer(\Omega_C)$. Then, for every formula $\psi(x)$ in  $For_1(\Omega_C)$ with (at most) a free variable $x$ it holds:\\[1mm]
(1) $[\forall x \psi]^C = \bigwedge_{\A_C} \{ [\psi[x/t]]^C \ :  \ t \in U\}$, where $\bigwedge_{\A_C}$ denotes  an existing infimum in the Boolean algebra $\mathcal{A}_C$;\\[1mm]
(2) $[\exists x \psi]^C = \bigvee_{\A_C} \{ [\psi[x/t]]^C \ :  \ t \in U\}$, where $\bigvee_{\A_C}$ denotes  an existing supremum in the Boolean algebra $\mathcal{A}_C$.
\end{lemma}
\begin{proof} \ \\
(1)
By the given definitions and by the rules from \cplp, $[\alpha]^C \leq [\beta]^C$  in $\A_C$ iff $\vdash_{\rqmbc}^C \alpha \rightarrow\beta$. Let $\psi (x)$ be a formula  in  $For_1(\Omega_C)$  with (at most) a free variable $x$. Then $[\forall x \psi]^C \leq [\psi[x/t]]^C$ for every $t \in  U$, by ${\bf (Ax\forall)}$. Let $\beta$ be a formula in $For_1(\Omega_C)$ such that $[\beta]^C \leq [\psi[x/t]]^C$ for every $t \in U$. Given that $\beta$ and $\psi(x)$ are finite expressions and $C$ is infinite, there exists a constant $c$ in $C$ such that $c$ neither occurs in $\beta$ nor in $\psi(x)$ and, by hypothesis,  $[\beta]^C \leq [\psi[x/c]]^C$. That is, $\vdash_{\rqmbc}^C \beta \rightarrow\psi[x/c]$. Let $d$ be a derivation   $\varphi_1\ldots \varphi_n$ of $\beta \rightarrow\psi[x/c]$ in \rqmbc\ over the signature $\Omega_C$ (hence $\varphi_n=\beta \rightarrow\psi[x/c]$), and let $z$ be a variable which does not occur (free or bounded) in the formulas of the derivation $d$. Observe that such variable always exists, since the number of variables occurring (free or bounded) in $d$ is finite.  Let 
$\varphi'_i$ be the formula obtained from $\varphi_i$ by substituting uniformly the constant $c$ by the variable $z$, for $1 \leq i \leq n$. Observe that $\varphi'_n=\beta \rightarrow\psi[x/z]$, given that $c$ occurs neither in $\beta$ nor in $\psi(x)$.
Since $z$ does not occur in $d$, it is free for $x$ in $\psi$ and it does not occur in $\beta$. By an argument similar to the one given in the proof of Proposition~\ref{constants}, the sequence $d'$  given by  $\varphi'_1\ldots \varphi'_n$ is a derivation of $\beta \rightarrow\psi[x/z]$ in \rqmbc\ over the signature $\Omega_C$, such that $z$ does not occur free in $\beta$. By applying the rule ${\bf (\forall\mbox{\bf -In})}$, it follows that $\vdash_{\rqmbc}^C \beta \rightarrow \forall z\psi[x/z]$. But $z$ is free for $x$ in $\psi$, hence  $\vdash_{\rqmbc}^C \forall z\psi[x/z] \rightarrow \forall x\psi$ and so $\vdash_{\rqmbc}^C \beta \rightarrow \forall x\psi$. That is, $[\beta]^C \leq [\forall x\psi]^C$. This shows that  $[\forall x \psi]^C = \bigwedge_{\A_C} \{ [\psi[x/t]]^C \ :  \ t \in U\}$.\\[1mm]
(2) It is proved analogously. 
\end{proof}

\begin{prop}\label{rem2}
For every sentence $\psi$ in $Sen(\Omega_C)$, $\termvalue{\psi}^{\mathfrak{A}_C} = \ast([\psi]^C)$. 
Moreover, $\termvalue{\psi}^{\mathfrak{A}_C}=1^C \mbox{ iff } \vdash_{\rqmbc}^C \psi$.
\end{prop}
\begin{proof}
The proof is done by induction on the complexity of the sentence $\psi$ in $Sen(\Omega_C)$. If $\psi=P(t_1,\ldots,t_n)$ is atomic then, by using Definition~\ref{val}, the fact that $\termvalue{t}^{\mathfrak{A}_C}=t$ for every $t \in CTer(\Omega_C)$, and Definition~\ref{str}, we have:\\[1mm]
$\termvalue{\psi}^{\mathfrak{A}_C} = P^{\mathfrak{A}_C}(\termvalue{t_1}^{\mathfrak{A}_C}, \ldots, \termvalue{t_n}^{\mathfrak{A}_C})=P^{\mathfrak{A}_C}(t_1, \ldots,t_n)=\ast([\psi]^C)$.\\[1mm]
If $\psi = \# \beta$ for $\# \in \{\neg, \cons\}$ then, by Definitions~\ref{val} and~\ref{nma} and by induction hypothesis,
$$\termvalue{\psi}^{\mathfrak{A}_C} \,=\, \#_\ast\termvalue{\beta}^{\mathfrak{A}_C} \,=\, \#_\ast(\ast([\beta]^C)) \,=\, \ast([\#\beta]^C).$$
If $\psi = \alpha \# \beta$ for $\# \in \{\wedge, \vee, \to\}$, the proof is analogous.\\[1mm]
If $\psi = \forall x \beta$ then, by Lemma~\ref{quantOK} and using that $U=CTer(\Omega_C)$,
$[\forall x \beta]^C = \bigwedge_{\A_C} \{ [\beta[x/t]]^C \ :  \ t \in U\}$ and so $\ast([\forall x \beta]^C) = \bigwedge_{\textsf{C}\A_C} \{ \ast([\beta[x/t]]^C) \ :  \ t \in U\}$. Let $\mu$ be an assignment over $\mathfrak{A}_C$. Since $\termvalue{t}^{\mathfrak{A}_C}_\mu=t$ for every $t \in U$ it follows, by Theorem~\ref{substlem} and by induction hypothesis, that $\termvalue{\beta}^{\mathfrak{A}_C}_{\mu^x_t} = \termvalue{\beta[x/t]}^{\mathfrak{A}_C}_\mu= \termvalue{\beta[x/t]}^{\mathfrak{A}_C}=\ast([\beta[x/t]]^C)$.
Then, by Definition~\ref{val}: 
$$
\termvalue{\forall x \beta}^{\mathfrak{A}_C} = \termvalue{\forall x \beta}^{\mathfrak{A}_C}_\mu = \bigwedge_{t \in U} \termvalue{\beta}^{\mathfrak{A}_C}_{\mu^x_t} = \bigwedge_{t \in U} \ast([\beta[x/t]]^C) = {\ast}([\forall x\beta]^C).$$

If $\psi = \exists x \beta$, the proof is analogous to the previous case. 

This shows that $\termvalue{\psi}^{\mathfrak{A}_C} =\ast([\psi]^C)$ for every sentence $\psi$. The rest of the proof follows by Proposition~\ref{balfi-well}.
\end{proof} 

	\begin{teom} [Completeness  of \rqmbc\ w.r.t. BALFI semantics] \ \\ \label{comple-rqmbc}
		For every $\Gamma \cup \{\varphi\} \subseteq$ \sent: if $\Gamma \models_{\rqmbc} \varphi$ then $\Gamma \vdash_{\rqmbc} \varphi$. 
	\end{teom}
	\begin{proof} 
		Let $\mathfrak{A}_C$ be as in Definition~\ref{str}, and let $\mathfrak{A}'$ be the reduct of $\mathfrak{A}_C$ to $\Omega$.\footnote{That is, $\mathfrak{A}'$ coincides with $\mathfrak{A}_C$, with the exception that it `forgets' to interpret the constants in $C$. Hence, $\mathfrak{A}'$ is a structure over $\Omega$.}
		Let $\Gamma \cup \{\varphi\} \subseteq$ \sent\ such that $\Gamma \models_{\rqmbc} \varphi$. 
		By Definition~\ref{consrel0},  we have two cases to analyze:\\[1mm]
		(1) $\varphi$ is  valid in \rqmbc. Then, $\varphi$ is valid in  $\mathfrak{A}'$.  It is routine to prove that, if $\psi \in For_1(\Omega)$ and $\mu \in \textsc{A}(\mathfrak{A}')=\textsc{A}(\mathfrak{A}_C)$, then $\termvalue{\psi}^{\mathfrak{A}'}_\mu=\termvalue{\psi}^{\mathfrak{A}_C}_\mu$. From this, $\varphi$  is valid in $\mathfrak{A}_C$. By Definition~\ref{consrel0}, $\termvalue{\varphi}^{\mathfrak{A}_C} =1^C$. Hence, $\vdash_{\rqmbc} \varphi$, by Propositions~\ref{rem2} and~\ref{constants}. Therefore,   $\Gamma\vdash_{\rqmbc} \varphi$, by definition of derivations from premises  in \rqmbc.\\[1mm]
		(2) There exists a finite, non-empty subset $\{\gamma_1,\ldots,\gamma_n\}$ of $\Gamma$ such that the closed formula $\psi=(\gamma_1 \wedge \gamma_2 \wedge \ldots\wedge \gamma_n) \to \varphi$ over $\Omega$ is  valid in \rqmbc. In particular, $\psi$ is valid in $\mathfrak{A}'$ and so, as observed in~(1), it is valid in $\mathfrak{A}_C$.
		But then, by item~(1), we have that $\vdash_{\rqmbc} \psi$, which implies that  $\Gamma\vdash_{\rqmbc} \varphi$, by definition of derivations from premises  in \rqmbc.
		
		This completes the proof.
	\end{proof}

\begin{rem}
The completeness result for \rqmbc\ w.r.t. BALFI semantics was obtained just for sentences, and not for formulas possibly containing free variables (as it was done with the soundness Theorem~\ref{adeq-rqmnc-BALFI}). This can be easily overcome. Recall that the {\em universal closure} of a formula $\psi$ in $For_1(\Omega)$, denoted by  $(\forall)\psi$, is defined as follows: if $\psi$ is a sentence then  $(\forall)\psi \defin \psi$; and if $\psi$ has exactly the variables $x_1,\ldots,x_n$ occurring free then  $(\forall)\psi \defin \forall x_1\cdots\forall x_n\psi$. If $\Gamma$ is a set of formulas in $For_1(\Omega)$ then  $(\forall)\Gamma \defin \{(\forall)\psi  \ : \  \psi \in \Gamma\}$. It is easy to show that,   for every $\Gamma \cup \{\varphi\} \subseteq For_1(\Omega)$:  (i)~$\Gamma\vdash_{\rqmbc} \varphi$ \ iff \   $(\forall)\Gamma\vdash_{\rqmbc} (\forall)\varphi$; and (ii)~$\Gamma \models_{\rqmbc} \varphi$ \ iff \   $(\forall)\Gamma \models_{\rqmbc} (\forall)\varphi$. From this, a general completeness for \rqmbc\ result follows from Theorem~\ref{comple-rqmbc}.
\end{rem}

\section{Conclusion, and  significance of the results}
 
 This paper offers a solution to  two open problems  in the domain of paraconsistency, in particular connected  to algebraization of  \lfis.  The quest for the  algebraic counterpart of paraconsistency is more than 50 years old: since the outset of da Costa's paraconsistent calculi,  algebraic   equivalents  for such systems  have been sought, with different degrees of success (and  failure).  
Our results suggest that the new concepts and  methods proposed  in the present paper, in particular the neighborhood-style semantics introduced for BALFIs, have a good potential for applications. As suggested  in~\cite{marcos:2004},  modal logics could alternatively be regarded as the study of  a kind of  modal-like contradiction-tolerant systems.  In alternative to founding  modal semantics  in  terms of belief, knowledge, tense, etc., modal logic could be regarded as a  general `theory of opposition',  more akin to the Aristotelian tradition.
 
Applications of paraconsistent logics in  computer science, probability and AI, just to mention a few areas, are greatly advanced when more traditional  algebraic tools pertaining to extensions  of Boolean algebras and  neighborhood semantics, are used to express the underlying ideas of paraconsistency. In addition, 
many logical systems employed in deontic logic and normative reasoning, where non-normal modal logics and neighborhood semantics play an important role, could be extended by means of our approach.
Hopefully, our results  may unlock new research in this direction.  Finally, BALFI semantics for \lfis\ opens the possibility of obtaining new algebraic models for paraconsistent set theory  \citep{CC13,CC19} by generalizing the well-known Boolean-valued models for ZF \citep{bell}.

\

\

 \noindent {\bf Acknowledgements:} 
The first and second authors acknowledge support from  the  National Council for Scientific and Technological Development (CNPq), Brazil
under research grants 307376/2018-4 and 306530/2019-8, respectively.
The third author acknowledges support from the Luxembourg National Research Fund (FNR), grant CORE AuReLeE - Automated Reasoning with Legal Entities  (C20/IS/14616644).

\newpage

\addcontentsline{toc}{section}{References}
\bibliographystyle{chicago}

\end{document}